\documentclass[11pt,a4paper]{amsart}

\usepackage[utf8]{inputenc}
\usepackage[T1]{fontenc} 
\usepackage[ngerman,english]{babel} 

\usepackage{mathrsfs}
\usepackage{mathdots}

\usepackage{csquotes} 
\usepackage{graphicx}
\usepackage{hyperref}
\usepackage{amssymb}
\usepackage{amsmath}
\usepackage{verbatim,enumerate}
\usepackage{tikz}
\usetikzlibrary{cd}
\usetikzlibrary{babel}

\numberwithin{equation}{section}
\numberwithin{figure}{section}

\newtheorem{thm}{Theorem}[section]

\newtheorem{lemma}[thm]{Lemma}
\newtheorem{prop}[thm]{Proposition}
\newtheorem{proposition}[thm]{Proposition}
\newtheorem{cor}[thm]{Corollary}

\newtheorem{defi}[thm]{Definition}
\newtheorem{definition}[thm]{Definition}
\theoremstyle{definition}
\newtheorem{exa}[thm]{Example}
\newtheorem{example}[thm]{Example}
\newtheorem{rem}[thm]{Remark}
\newtheorem{remark}[thm]{Remark}
\newtheorem*{ack}{Acknowledgments}

\newcommand{\C}{\mathbb{C}}
\newcommand{\TT}{\mathbb{T}}
\newcommand{\Fbold}{\mathbf{F}}

\newcommand{\RP}{\mathbb{R}\mathbf{P}}
\newcommand{\F}{\mathcal{F}}
\newcommand{\G}{\mathcal{G}}

\renewcommand{\epsilon}{\varepsilon}
\newcommand{\R}{\mathbb{R}}

\newcommand{\Z}{\mathbb{Z}}

\newcommand{\D}{\mathcal{D}}

\newcommand{\BM}{{\text{BM}}}
\newcommand{\fund}{\Omega}

\newcommand{\EE}{\mathcal{E}}
\newcommand{\EEE}{\mathcal{E}}
\newcommand{\CCC}{\mathcal{C}}
\newcommand{\SSS}{\mathcal{S}}

\renewcommand{\F}{\mathcal F}
\renewcommand{\G}{\mathcal G}
\newcommand{\FF}{\mathbb F_2}

\renewcommand{\geq}{\geqslant}
\renewcommand{\ge}{\geqslant}
\renewcommand{\leq}{\leqslant}
\renewcommand{\le}{\leqslant}

\newcommand{\bary}{\operatorname{Bar}}
\newcommand{\Card}{\operatorname{Card}}

\newcommand{\Hom}{\operatorname{Hom}}
\newcommand{\last}{\operatorname{last}}
\newcommand{\first}{\operatorname{first}}
\newcommand{\Vol}{\operatorname{Vol}}
\newcommand{\bw}{\bigwedge\nolimits}
\newcommand{\Fac}{\operatorname{Fac}}

\newcommand{\rk}{\operatorname{rk}}

\begin{document}

\title{Combinatorial patchworking: back from tropical geometry}

\author{Erwan Brugall\'e}
\address{Erwan Brugall\'e, Universit\'e de Nantes, Laboratoire de
  Math\'ematiques Jean Leray, 2 rue de la Houssini\`ere, F-44322 Nantes Cedex 3,
France}
\email{erwan.brugalle@math.cnrs.fr}

\author{Luc\'ia L\'opez de Medrano}
\address{Luc\'ia L\'opez de Medrano, Unidad Cuernavaca del Instituto de Matemáticas, UNAM,
Mexico}
\email{lucia.ldm@im.unam.mx}

\author{Johannes Rau}
\address{Johannes Rau, Universidad de los Andes, Carrera 1 no.~18A-12, Bogotá, Colombia.}
\email{j.rau@uniandes.edu.co}


\begin{abstract}
We show that,
once translated to the dual setting of convex
 triangulations of lattice polytopes,
results and  methods from
\cite{ArnRenSha18}, \cite{RenSha18}, \cite{JRS-Lefschetz11Theorem} and \cite{RenRauSha22} 
extend to non-convex
triangulations.
So, while the translation of Viro's 
patchworking method to the setting 
of tropical hypersurfaces
has inspired several tremendous 
developments over the last two decades,
we return to the original 
polytope setting 
in order to generalize and simplify 
some results regarding the topology of $T$-submanifolds of real toric varieties. 
\end{abstract}

\maketitle
\hfill\textit{
  \begin{tabular}{l}
Bourré de complexes\\ Et tout a changé
  \end{tabular}}

\hfill Boris Vian
\vspace{3ex}

\tableofcontents

\subsection*{Notation}
Throughout the text $\FF$ denotes the field $\Z/2\Z$.
If $X$ is  a topological space, we denote by $b_i(X)$ the $i$-th Betti
number of $X$ with coefficient in $\FF$.
A lattice
polytope  $\Delta\subset \R^n$ is a convex polytope
with vertices in $\Z^n$. 
Such a polytope defines a complex toric
variety $\mbox{Tor}_\C(\Delta)$, and $\Delta$ is said to be
non-singular if it is full-dimensional and if $\mbox{Tor}_\C(\Delta)$ is non-singular.
In this case, we
denote by $h^{p,q}(\Delta)$ the corresponding Hodge number of a
non-singular hypersurface in $\mbox{Tor}_\C(\Delta)$ with Newton
polytope $\Delta$.
More generally, we denote by
$h^{p,q}(\Delta^k)$  the $(p,q)$-Hodge number of a
non-singular complete intersection of $k$
hypersurfaces in $\mbox{Tor}_\C(\Delta)$ with Newton
polytope $\Delta$.
By convention such a complete intersection is
defined to  be $\mbox{Tor}_\C(\Delta)$ when $k=0$.
The standard complex
conjugation on $\C^*$ induces a real structure on
 the complex algebraic variety  $\mbox{Tor}_\C(\Delta)$,
 whose real part is
 denoted by  $\mbox{Tor}_\R(\Delta)$.

\section{Introduction}

\subsection{Betti numbers of $T$-manifolds}\label{sec:patch1}

Let us start by briefly recalling Viro's combinatorial patchworking for
projective
hypersurfaces. Let $\Delta_{n}\subset\R^n$ be the standard $n$-simplex
with vertices
\[
0,\quad (1,0,0,\cdots,0),\quad (0,1,0,\cdots,0),\quad \cdots,\quad
(0,0,\cdots,0,1). 
\]
Choose an integer $d\ge 1$, a  triangulation $\Gamma$ of
the $d$-dilation $d\Delta_n$ of $\Delta_n$, and
a sign distribution on $d\Delta_n$, that is to
say a function
$\epsilon:d\Delta_n\cap\Z^n \to \FF$ (see Figure \ref{fig:patch}a).
We denote by $\widetilde{d\Delta}_n$ and
$\widetilde \Gamma$ the union of all copies of $d\Delta_n$ and
$\Gamma$ under successive orthogonal symmetries with respect to
coordinate hyperplanes of $\R^n$.
Extend also the sign distribution $\epsilon$  to
$\widetilde{d\Delta}_n\cap\Z$ using the following rule
(see Figure \ref{fig:patch}b):
given $(i_1,\cdots, i_n)\in d\Delta_n\cap\Z^n$ and
$(s_1,\cdots,s_n)\in \FF^n$, define
\[
\epsilon\big((-1)^{s_1}i_1,\cdots, (-1)^{s_n}i_n\big)=(-1)^{s_1i_1+\cdots
  +s_ni_n}\
\epsilon(i_1,\cdots,i_n).
\]
\begin{figure}[h]
\centering
\begin{tabular}{ccccc}
  \includegraphics[height=5cm, angle=0]{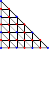}
  & \hspace{1ex} &
  \includegraphics[height=5cm, angle=0]{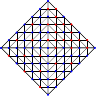}
  & \hspace{1ex} &
  \includegraphics[height=5cm, angle=0]{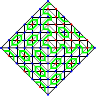}
\\  a) & & b) && c)
\end{tabular}
\caption{Example of patchworking of $6\Delta_2$;
  the sign distribution $\epsilon$ takes value 0 on
  \textcolor{blue}{$\bullet$}-points
  , and 1 on
  \textcolor{red}{$\bullet$}-points.}\label{fig:patch}
\end{figure}
In each simplex $\sigma$
of $\widetilde \Gamma$, separate the vertices with
different signs by taking the convex hull of all middle points of 
edges of $\sigma$ having different signs (see Figure \ref{fig:patch}c). Note that this convex hull
is either empty (if $\epsilon$ is constant on $\sigma$), or is a convex
polyhedron of dimension $n-1$ contained in $\sigma$.
Denote by 
$X^o_{\Gamma,\epsilon}$ the obtained $PL$-hypersurface of 
$\widetilde{d\Delta}_n$.
Finally identify, via the antipodal map,
opposite pairs of points on the boundary of
$\widetilde{d\Delta}_n$, and denote by $X_{\Gamma,\epsilon}$ the obtained
$PL$-hypersurface of the real projective space
$\RP^n$.
It is called a \emph{$T$-hypersurface}.

The subdivision $\Gamma$ is called \emph{convex}, or \emph{regular},
or \emph{coherent}, if there exists
a convex piecewise linear function $\lambda:d\Delta_n\to \R$ whose
domains of linearity are exactly the  facets of $\Gamma$.
When $\Gamma$ is convex,  Viro's combinatorial patchworking Theorem
\cite{V1,V2,IV2,GKZ} states that the $T$-hypersurface $X_{\Gamma,\epsilon}$ is
isotopic in $\RP^n$ to the set of real solutions of the polynomial equation
\[
\sum_{(i_1,\cdots,i_n)\in \text{Vert}(\Gamma)}
(-1)^{\epsilon(i_1,\cdots,i_n)}t^{\lambda(i_1,\cdots,i_n)
}x_1^{i_1}x_2^{i_2}\cdots x_n^{i_n}x_{n+1}^{d-i_1-\cdots -i_n}=0,
\]
where $\mbox{Vert}(\Gamma)$ denotes the set of vertices of the
triangulation $\Gamma$,
for $t$ a small enough positive real number.
Hence, this beautiful result establishes a bridge between
real algebraic geometry and combinatorics.
Nevertheless, 
the construction of  $X_{\Gamma,\epsilon}$ makes
sense even if $\Gamma$ is not convex, 
in which case the relations of
$X_{\Gamma,\epsilon}$ to real algebraic geometry remains  unclear.
We refer the interested reader to
\cite{Loe,IS,IS2,Br3,Br4} for some investigations in this direction.

We restricted ourselves so far to projective hypersurfaces for simplicity,
nevertheless the  construction of $T$-hypersurfaces  actually makes sense for any
lattice polytope $\Delta$ (only the gluing rule for the $2^n$
copies of $\Delta$ has to be adapted, see Section
\ref{sec:realstructures}). One then obtains
a $PL$-hypersurface  $X_{\Gamma,\epsilon}$ in
$\mbox{Tor}_\R(\Delta)$ out of a triangulation $\Gamma$ of $\Delta$
and a sign distribution on $\Delta$,
 and Viro's
combinatorial patchworking Theorem still holds true when $\Gamma$ is convex.

In connection to Hilbert's 16th problem,
one is thus interested in exploring the range of possibilities
for the topology of
$X_{\Gamma,\epsilon}$. 
This paper fits into this perspective.
While the general situation remains poorly understood, Itenberg
conjectured almost 20 years ago
that for convex and \emph{unimodular} triangulations, Betti numbers of
$T$-hypersurfaces are bounded by diagonal sums of
Hodge numbers of the corresponding
projective hypersurfaces.
A triangulation $\Gamma$ is said to be unimodular, or \emph{primitive}, if
all facets of $\Gamma$ have lattice volume 1, or
equivalently  are the standard $n$-simplex $\Delta_n$ up to translations
and the action
of $GL_n(\Z)$.
Itenberg's conjecture has recently been proved by Renaudineau and Shaw
in \cite{RenSha18}. In this note we show that
the approach by Renaudineau and Shaw extends to non-convex triangulations.

\begin{thm}\label{thm:main hyp}
Let $\Delta\subset \R^n$ be a non-singular lattice polytope, let
$\Gamma$ be a unimodular triangulation of $\Delta$, and let
$\epsilon$ be a sign 
distribution on $\Delta$. Then one has
\[
\forall p\ge 0,\qquad b_p(X_{\Gamma,\epsilon})\le \sum_{q\ge 0}h^{p,q}(\Delta).
\]
\end{thm}

When the subdivision $\Gamma$ is convex, this statement is exactly
\cite[Theorem 1.4]{RenSha18}.
Theorem \ref{thm:main hyp} has first been proved independantly  by
Haas \cite{Haa2} and Itenberg \cite{Ite95}
for $T$-curves,
and by Itenberg  \cite{Ite97} for $T$-surfaces in $\RP^3$.
Renaudineau-Shaw's results \cite{RenSha18} sits at the crossroad
  of algebraic geometry and combinatorics, and Theorem
  \ref{thm:main hyp} provides a combinatorial generalization. An
 algebro-geometric
  generalization of \cite[Theorem 1.4]{RenSha18} has recently been proposed by
  Ambrosi and Manzaroli in \cite{AmbMan22}.

\medskip
In this paper we actually study not only $T$-hypersurfaces, but
\emph{$T$-manifolds} of arbitrary codimension in arbitrary
non-singular real toric manifolds.
To do so, the data of a sign distribution 
is replaced by a \emph{real phase structure}  
on the $k$-skeleton of
$\Gamma$
in analogy to the real phase structures studied in 
\cite{RenRauSha21}.
 Given such real phase structure $\EE$, we
 construct a $PL$-manifold  $X_{\Gamma,\EE}$, called a $T$-manifold,
of codimension $k$ in
$\mbox{Tor}_\R(\Delta)$. We refer to Section
\ref{sec:realstructures} for precise definitions.
The next theorem is the main result of this paper, and
specializes to Theorem \ref{thm:main hyp} when $k=1$.

\begin{thm}\label{thm:main}
Let $\Delta\subset \R^n$ be a non-singular lattice polytope, let
$\Gamma$ be a unimodular triangulation of $\Delta$, and let $\EE$ be a
real phase structure on the $k$-skeleton of
$\Gamma$. Then one has
\[
\forall p\ge 0,\qquad b_p(X_{\Gamma,\EE})\le \sum_{q\ge 0}h^{p,q}(\Delta^k),
\]
with equality when $k=0$ or $k=n$.
\end{thm}

Again, when $\Gamma$ is convex, there already exists a tropical counterpart:
In this case, the data of $\EE$ corresponds to 
a real phase structure, in
the sense of \cite[Section 2]{RenRauSha21}, on the $k$-th stable
intersection of a tropical hypersurface dual to $\Gamma$.
Our $T$-manifold corresponds to the associated \emph{real patchwork},
whose properties are discussed for general non-singular tropical varieties with real
phase structures in \cite{RenRauSha22}.
In particular, Theorem \ref{thm:main}  is then a straightforward combination of
\cite[Theorem 1.2, Section 2.6]{RenRauSha22} with Theorem \ref{thm:hodge} below.

We point out that
$T$-manifolds do not seem to be related to the
generalization to complete intersections by Sturmfels \cite{Stur94}
 of Viro's
patchworking Theorem. Sturmfels's
construction produces a $PL$-manifold of codimension $k$ out of
a
\emph{mixed subdivision} of
$k$  triangulations of  $\Delta$.
This $PL$-manifold depends
not only on the initial triangulations but also heavily 
on the chosen mixed subdivision.
Here our construction of
$T$-manifolds does not make use of any
mixed subdivision, but restricts in
 return to considering $k$ times the same subdivision of $\Delta$.
 In particular we do not know, even when $\Gamma$ is convex,
 whether a $T$-manifold in
 $\mbox{Tor}_\R(\Delta)$ is always isotopic to the real part of
 a complete intersection
 of real algebraic hypersurfaces
 in $\mbox{Tor}_\R(\Delta)$ with Newton polytope $\Delta$.
So, it might be interesting to study this further and to compare 
the  range of topological possibilities of both constructions. 

Finally, note that Theorem \ref{thm:main} has been known for a long
time in the cases $k=0$ or $k=n$. In the case $k=n$, Theorem
\ref{thm:main} simply states that the numbers of facets of $\Gamma$
is the lattice volume of $\Delta$, which follows from the definition of
unimodularity of $\Gamma$.
When $k=0$ there exists a unique real phase structure, and in this
case $X_{\Delta,\EE}=\mbox{Tor}_\R(\Delta)$.
Hence Theorem
\ref{thm:main} is now a consequence of \cite[Section 4]{BFMV06}.

\medskip
As in
\cite{RenSha18}, we obtain as
a by-product of our proof of Theorem \ref{thm:main}
   the following relation between the Euler
characteristic of $X_{\Gamma,\EE}$, and the topological signature
$\sigma(\Delta^k)$ of a
non-singular complete intersection of $k$
hypersurfaces in $\mbox{Tor}_\C(\Delta^k)$ with Newton
polytope $\Delta$.

\begin{thm}\label{thm:chi}
  Let $\Delta\subset \R^n$ be a non-singular lattice polytope, let
$\Gamma$ be a unimodular triangulation of $\Delta$, and let $\EE$ be a real structure  on the $k$-skeleton of
$\Gamma$. Then one has
\[
\chi(X_{\Gamma,\EE})=\sigma(\Delta^k).
\]
\end{thm}

When $k=1$,
Theorem \ref{thm:chi} has originally been proved by Itenberg \cite{Ite97}
when $n=3$, and  generalized by 
Bertrand \cite{Ber2} for any $n$.
When $\Gamma$ is convex and $k=1$, Renaudineau and
Shaw gave in \cite{RenSha18}
an alternative proof of Theorem \ref{thm:chi},
later extended
to higher codimensions by Renaudineau, Rau, and Shaw \cite{RenRauSha22}.
All these proofs are combinatorial and do not use the relation of
combinatorial patchworking to real algebraic geometry given by Viro's
Theorem. An algebro-geometric proof of Theorem \ref{thm:chi}
has been proposed by
the first author in \cite{Bru22} for hypersurfaces and convex triangulations.

As mentioned above, $T$-manifolds of higher
codimensions are a priori not related to
Sturmfels' combinatorial patchworking
of complete intersections. In particular when $k\ge 2$, and
even if $\Gamma$ is convex, Theorem \ref{thm:chi} seems to be 
disjoint from the  results in \cite{BerBih07} and \cite{Bru22}.

\subsection{Hodge numbers, Poincaré duality, and Heredity}
As mentioned in the beginning of this introduction,
this paper is built upon the observation that
the aforementioned \enquote{tropical} works 
do not seem to use in an essential way  the convexity of the triangulations
dual to the tropical hypersurfaces under study.
This will probably seem obvious to experts once formulated, yet it
requires some efforts to rigorously prove this observation. Indeed
 the above works refer at some points to former 
results about abstract non-singular tropical manifolds (e.g. Poincaré
duality \cite{JRS-Lefschetz11Theorem}), and explicitly use convexity
arguments at some steps of their reasoning (e.g. in the computation of
tropical $p$-characteristics in \cite[Proof of Theorem 1.8]{ArnRenSha18}). 
Altogether, our task is to translate
the aforementioned results
to the dual setting  of unimodular
triangulations
and to free them from the convexity hypothesis 
(in particular, from any reference to tropical manifolds).

Following \cite{RenSha18,RenRauSha22}, 
the proof of Theorem \ref{thm:main}
is based on the computation of homology groups of two families of
\emph{combinatorial cosheaves}
 defined on the poset $\Xi$ of \emph{cell pairs} of $\Gamma$.
Such a cell pair is defined
to be  a couple $(F,\sigma)$ with
$F$  a face of $\Delta$ and $\sigma$  a face of $\Gamma$ contained
in $F$.
A cosheaf $\F$ on $\Xi$ is the data of an
$\FF$-vector space $\F(F,\sigma)$ 
associated to each cell pair $(F,\sigma)$ of $\Gamma$,
 satisfying some compatibility relations.
We refer to 
Section \ref{sec:hom} for a precise definition, as well
as for the definition 
of  homology groups
$H_{q}(\Gamma;\F)$ of $\F$. The rank
of $H_{q}(\Gamma;\F)$, as an $\FF$-vector space, is denoted by
$h_{q}(\Gamma;\F)$.

In a first step of the proof of Theorem \ref{thm:main}, we consider
for each integers $p$ and $k$ a cosheaf $\F_p^k$ on $\Xi$
that recovers the Hodge numbers
$h^{p,q}(\Delta^k)$. 
Given a rational polyhedron $\sigma \subset \R^n$, we  denote by
$T(\sigma)$ the tangent space of $\sigma$ (that is, the linear space generated
by the vectors $v-w$, $v,w \in \sigma$)
and set $T_\Z(\sigma) = T(\sigma) \cap \Z^n$ and 
$T_{\FF}(\sigma) = T_\Z(\sigma) \otimes \FF$. 
Moreover, we use 
$\sigma^\perp$ as shorthand for $T_{\FF}(\sigma)^\perp \subset (\FF^n)^\vee$.
Now, for any $p,k\ge 0$, the cosheaf $\F_p^k$ on $\Xi$ is defined by
  \[
  \forall (F,\sigma)\in\Xi,\quad
  \F_p^k(F,\sigma)=\sum_{\substack{\tau\subset \sigma\\\dim \tau=k}}
  \bw^p \left(\tau^\perp/F^\perp \right).
  \]
When $\Gamma$ is convex, 
these cosheaves have been introduced in the
dual setting of tropical subvarieties by
Itenberg, Katzarkov, Mikhalkin and Zharkov in \cite{IKMZ19}.
More precisely, $\F_p^k$ corresponds to the cosheaf of framing groups $\F_p$ of
the $k$-th stable intersection $X^k$ of a tropical hypersurface $X$ dual to $\Gamma$. 
In particular, 
a result in \cite{IKMZ19} shows that 
when $\Gamma$ is convex we have $h_{q}(\Gamma;\F_p^k) = h_{q}(X^k;\F_p)$.

The next theorem is the second main statement of our paper.

\begin{thm}\label{thm:hodge}
  For any $p,q$ and $k$,  one has
  \[
  h_{q}(\Gamma;\F_p^k)=h^{p,q}(\Delta^k).
  \]
\end{thm}

When $\Gamma$ is convex and $k=1$, Theorem \ref{thm:hodge} has first
been proved by Arnal, Renaudineau, and
Shaw  in \cite{ArnRenSha18}.

\begin{rem}
  We only consider $\FF$-cosheaves in this paper, since
	coefficients in $\FF$ are sufficient for the purposes of Theorem \ref{thm:main}.
	Nevertheless a large part of the text, in particular Theorem \ref{thm:hodge}
	and Theorem \ref{thm:propertylist}, can be extended to their obvious $\Z$-cosheaf
  version $\F_{p,\Z}^k$ as in \cite{ArnRenSha18, JRS-Lefschetz11Theorem}, or $\F_{p,K}^k$
	for an arbitrary field $K$. 
	See also Remark \ref{rem:Z} below.
\end{rem}

  At this point, it may be worthwhile to recall that
  all Hodge numbers $h^{p,q}(\Delta^k)$ are known since the seminal
  work \cite{DanKho86} by Danilov and Khovanski\u{\i}.
  Since we will follow an analogous strategy in the proof of Theorem
  \ref{thm:hodge}, we briefly indicate the procedure to compute
  these numbers.
  
  As a consequence of Lefschetz hyperplane section theorem, one obtains
  that
  \[
  h^{p,q}(\Delta^k)=h^{p,q}(\mbox{Tor}_\C(\Delta))\qquad
  \mbox{if }p+q<n-k.
  \]
	We refer to this property as Heredity in the following. 
  Combining this with  Danilov-Jurkiewicz
  theorem and Poincaré duality, one obtains all Hodge numbers
  $h^{p,q}(\Delta^k)$ except
   when $p+q=n-k$.
  In particular
    $h^{p,q}(\Delta^k)=0$ except possibly when
  $p=q$ or $p+q=n-k$. Now the computation of the numbers
  $h^{p,n-k-p}(\Delta^k)$ follows from the computation of the
  $p$-characteristics
  \[
  e_{\Delta,k,p}= \sum_{q\ge 0}(-1)^q h^{p,q}(\Delta^k).
  \]
Using motivic properties of  $e_{\Delta,k,p}$, Danilov and
Khovanski\u{\i} computed explicitely all  $p$-characteristics of
hypersurfaces of toric varieties,
and indicated an
algorithm to recursively compute $p$-characteristics of complete intersections
by reducing to the case of hypersurfaces. 
Following this algorithm, Di Rocco, Haase,
and Nill obtained
in \cite{DiRHaaNil19} a closed expressions for
$e_{\Delta,k,p}$ in terms of $k,p,$ and $\Delta$.
In conclusion all Hodge numbers $h^{p,q}(\Delta^k)$ are expressed in terms of
Hodge numbers of the ambient toric variety
$h^{p,q}(\mbox{Tor}_\C(\Delta))$, for which several
different calculations are available (see for example
\cite{Dan78,FMSS-IntersectionTheorySpherical,Jor-HomologyCohomologyToric}).

Hence, in order to prove Theorem \ref{thm:hodge}
we follow the strategy from complex algebraic geometry by
Danilov and Khovanski\u{\i}
which was also used in \cite{ArnRenSha18}: 
we prove that the numbers $h_{q}(\Gamma;\F_p^k)$
satisfy Heredity and Poincaré duality and
have the same $p$-characteristic as $h^{p,q}(\Delta^k)$.
For convenience, we state this here in summary
as our last main result.

\begin{thm}\label{thm:propertylist}
  The numbers $h_{q}(\Gamma;\F_p^k)$ satisfy the following properties.
	\begin{enumerate}
		\item (Heredity, Proposition \ref{prop:her}) 
			For $p+q<n-k$,
			\[ h_{q}(\Gamma;\F_p^k)=h_{q}(\Gamma;\F_p^0). \]
		\item (Poincaré duality, Theorem \ref{prop:dual}) 
			For any $p,q$ and $k$,
			\[ h_{q}(\Gamma;\F_p^k)=h_{n-k-q}(\Gamma;\F_{n-k-p}^k). \]
		\item ($p$-characteristic, Proposition \ref{prop:eq euler})
			For any $p$ and $k$,
			\[
			\sum_{q\ge 0} (-1)^qh_q(\Gamma;\F_p^k)=
			\sum_{q\ge 0} (-1)^qh^{p,q}(\Delta^k).
			\]
	\end{enumerate}
\end{thm}

When $\Gamma$ is convex and $k=1$, the items (Heredity) and ($p$-characteristic)
are proven in \cite{ArnRenSha18}. When $\Gamma$ is convex, item 
(Poincaré duality) is proven in \cite{JRS-Lefschetz11Theorem}.

\subsection{Real phase structures}

Given Theorem \ref{thm:hodge}, 
the second step in the proof of Theorem \ref{thm:main} is the
definition of real phase structures on the $k$-skeleton of
$\Gamma$. 
This is the translation 
to the dual setting of the real phase structure
for tropical varieties introduced in \cite{RenRauSha21}.
Given $\sigma$ a face of $\Gamma$, we define
$\sigma^\vee=(\FF^n)^\vee/\sigma^\perp$, and denote by
$\pi_\sigma:(\FF^n)^\vee\to \sigma^\vee$ the projection map. More
generally if $\sigma\subset\tau$, there is a natural projection map
$\pi_{\sigma,\tau}:\tau^\vee\to \sigma^\vee$.
\begin{defi}
   A \emph{real phase structure} $\EEE$ on the $k$-skeleton of
   $\Gamma$
   consists in a choice of a point
	\[
	  \EEE(\sigma) \in  \sigma^\vee
	\]
	for every $\sigma \in \Gamma$ of dimension $k$ such that
	for any $\tau \in \Gamma$ of dimension $k+1$ and any $s \in \tau^\vee$
	the set
	\[
	  \{\sigma \subset \tau : \dim(\sigma) = k \text{ and } \pi_{\sigma,\tau}(s) = \EEE(\sigma) \}
	\]
	has even cardinal.	
\end{defi}

Since $\sigma^\vee=\{0\}$ if $\dim \sigma=0$, there exists a unique
real phase structure on the vertices of $\Gamma$ ($k = 0$).
On the other extreme ($k = n$),
real phase structures on the $n$-skeleton of $\Gamma$ are also easy
to describe: they consist in an arbitrary choice of an element in
$(\FF^n)^\vee$ for all $n$-simplex of $\Gamma$.
Real phase structures on  edges of $\Gamma$ ($k = 1$)
provide an equivalent way to describe  sign distributions on
vertices of $\Delta$, see
Example \ref{ex:rps}.

Starting from a real phase structure $\EE$ on the $k$-skeleton of
$\Gamma$, we construct the $T$-manifold $X_{\Gamma,\EE}$ of codimension
$k$ in $\mbox{Tor}_\R(\Delta)$, and a \emph{sign cosheaf} $\SSS$ on $\Xi$.
We refer to Definition \ref{def:Tamnifold} for details.
As an example, we depicted in Figure \ref{fig:patch dte}
a $T$-line in $\RP^3$ (see Example \ref{ex:patch dte}).
\begin{figure}[h]
\centering
\begin{tabular}{c}
  \includegraphics[height=6cm, angle=0]{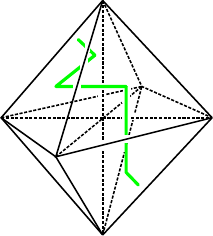}
\end{tabular}
\caption{A $T$-line in $\RP^3$.}\label{fig:patch dte}
\end{figure}
The sign cosheaf $\SSS$ on $\Xi$ is defined by
\[
\SSS(F,\sigma)=\FF^{\EE(F,\sigma)},
\]
where 
\[
\EE(F,\sigma)= \{ s \in F^\vee 
        : \pi_\tau(s) = \EEE(\tau) \text{ for some } \tau \subset
        \sigma\mbox{ with }\dim(\tau) = k\}.
\]
The sign cosheaf has first been defined for convex primitive
triangulations
in \cite{RenSha18} for $k=1$, and in \cite{RenRauSha22} for $k\ge 2$.
The proof of Theorem \ref{thm:main} now follows from the fact the
homology groups of $\SSS$ are related to both the homology groups of
$X_{\Gamma,\EE}$ and the homology groups of
$\F_{p}^k$. This is the fundamental idea from \cite{RenSha18} for
hypersurfaces,
later generalized in  \cite{RenRauSha22} in any codimension.
Here our task is simply to check that the two important statements about $\SSS$
from \cite{RenSha18,RenRauSha22} carry over to 
possibly non-convex triangulations.
First, the groups $H_p(\Xi; \SSS)$ and $H_p(X_{\Gamma, \EEE}; \FF)$ are
canonically isomorphic for any $p$, see Proposition \ref{prop:signcosheafhomology}.
Second, there exists a filtration 
of cosheaves
	\[
	  0 = \SSS_{n+1} \subset \SSS_n \subset \dots \subset \SSS_0 = \SSS
	\]
	such that $\SSS_p / \SSS_{p+1} \cong \F_p^k$, see Proposition \ref{prop:filtration}. 
From the spectral sequence associated to this filtration it follows
that 	
\begin{equation*}\label{eq:ineq}
h_q(\Gamma;\SSS)\le \sum_{p\ge 0}h_q(\Gamma;\F_p^k).
\end{equation*}
for any $p$. 
This completes the proof of Theorem \ref{thm:main}.

\begin{ack}
We are grateful to Arthur Renaudineau and Kris Shaw
for helpful comments and discussions on a preliminary version of the text. 

Part of this work has been achieved during the visit of J.R. and
L. L.d.M. at Nantes
Université. J.R. was funded by 
 the program \emph{Missions Chercheurs Invités} of Nantes Université,
 and L. L.d.M. was 
 funded by  ECOS NORD 298995, CONACyT 282937, CONACyT I1200/381/2019 and PAPIIT-IN108520. We
 thank  Laboratoire de Mathématiques Jean Leray
 for excellent working
 conditions. The author also
 thanks the France 2030 framework porgramme Centre Henri Lebesgue
 ANR-11-LABX-0020-01 for creating an attractive mathematical environment. 
  
 E. B.  is
partially supported by the grant TROPICOUNT of Région Pays de la
Loire, and the ANR project ENUMGEOM NR-18-CE40-0009-02. 
 J. R. acknowledges support from the FAPA grant by the Facultad de Ciencias, Universidad de los Andes, Bogotá.
\end{ack}

\section{Preliminaries}\label{sec:prel}

\subsection{Poset homology}\label{sec:hom}

A \emph{poset} is a set $P$ equipped with a partial order $\leq$. In this text, 
all posets will be finite. A \emph{cover relation}, denoted by $x \lessdot y$, is a pair $x < y$ such that
there exists no $z \in P$ with $x < z < y$.
A \emph{grading} of $P$ is a
function $\dim \colon P \to \Z$ such that $\dim(y) - \dim(x) = 1$ for
every cover relation
$x \lessdot y$.
Given a poset $P$, we denote by $P^{\text{op}}$ the poset with inverted partial order. 
If $P$ is graded by the function $\dim$, we equip $P^{\text{op}}$ with the grading $-\dim$.

Given a pair $x \leq y$, we denote by $[x,y] = \{z \in P : x \leq z \leq y\}$ the 
\emph{interval} between $x$ and $y$. We call $\dim(y) - \dim(x)$ the \emph{codimension}
of $(x,y)$ or the \emph{length} of $[x,y]$. 
A graded poset is \emph{thin} or \emph{satisfies the diamond property} if every
interval of length $2$ contains exactly $4$ elements. Such an interval is called a \emph{diamond}
of $P$ since schematically it looks as follows. 
\[
\begin{tikzcd}[cramped, sep=small]
     & y & \\
		z_1 \arrow[draw=none]{ru}[sloped,auto=false]{<} & & z_2 \arrow[draw=none]{lu}[sloped,auto=false]{>} \\
		& x \arrow[draw=none]{lu}[sloped,auto=false]{>} \arrow[draw=none]{ru}[sloped,auto=false]{<} & 
\end{tikzcd}
\]
Clearly, if $P$ is thin then $P^{\text{op}}$ is thin as well. 
The term \emph{thin} in this context was apparently coined by Björner
in \cite{Bjoe-PosetsRegularCw}. We refer to this work for more background and motivation. 

\medskip
Let $R$ be a ring. A \emph{$R$-cosheaf} $\F$ on a poset $P$ is a contravariant
functor from $P$ (whose morphisms are the ordered pairs $x \leq y$) to the category of $R$-modules. 
Analogously, a \emph{sheaf} on $P$ is given by a covariant functor. Since we will mostly work with cosheaves, let us spell the definition out in this case: we assign an $R$-module $\F(x)$ to any $x \in P$,
and have $R$-linear maps
$\iota_{x,y} \colon \F(y) \to \F(x)$ for every pair $x \leq y$ such
that 
\begin{equation} \label{eq:cosheaf} 
  \iota_{x,y} \circ \iota_{y,z} = \iota_{x,z}  
\end{equation}
for all triples $x \leq y \leq z$, and 
$  \iota_{x,x}= Id_{\F(x)}$ for any $x\in P$.

Let $P$ be a thin graded poset 
and $\F$ a $\FF$-cosheaf on $P$.
We consider the differential complex $C_\bullet(P;\F)$ given by 
\[
C_q(P; \F) = \bigoplus_{\dim (x) = q} \F(x),
\quad \quad \quad \partial: C_q(P; \F) \to C_{q-1}(P; \F),
\]
where the restriction of $\partial$ on $\F(x)$ is given by
\[
  \partial(\alpha)=\sum_{y \lessdot x} \iota_{y,x}(\alpha).
\]
Note that $\partial^2=0$ thanks to
the
diamond property of $P$ and the fact that we work over $\FF$. The associated homology groups are
denoted by $H_q(P; \F)$. 

\begin{rem} \label{rem:Z}
  For simplicity, in this text we restrict ourselves to cosheaves over $\FF$. 
	However, most of the 
	cosheaves and statements that we will encounter have
	analogous versions over $\Z$ or the reader's favourite ring. 
	In this generality, we additionally have to equip the poset $P$ with a \emph{balanced 
	signature} (or \emph{balanced colouring}). Here, a \emph{signature} is a map $s \colon 
	\CCC(P) \to \{+1, -1\}$ where $\CCC(P)$ denotes the set of all cover relations in $P$. 
	A signature is called \emph{balanced} if any diamond contains an odd number of $-1$'s. 
	Then the modified map $\partial = \sum_{x \lessdot y} s(x \lessdot y) \iota_{y,x}$
	still satisfies $\partial^2 =0$. It is easy to check that all the posets that
	are used in the following admit a balanced signature (obtained by choosing 
	orientations for the geometric objects in the background). 
	For more details on balanced signatures and poset homology in this generality we refer to 
	\cite[Section 2.7]{BB-CombinatoricsCoxeterGroups} and 
	\cite{Cha-ThinPosetsCw}.
\end{rem}

Given a poset $P$, a subset $U \subset P$ is \emph{open} if it is closed under 
taking larger elements, that is, 
\[
  x \in U, x < y \quad \Longrightarrow \quad y \in U.
\]
If $P$ is thin and $U \subset P$ is open, then $U$ (with the restricted partial order)
is thin as well. 
Given a cosheaf $\F$ on $P$, we denote the restriction to $U$ by $\F|_U$. 
For simplicity, we write $C_q(U; \F)$ and $H_q(U; \F)$ instead of 
$C_q(U; \F|_U)$ and $H_q(U; \F|_U)$, respectively, for the restricted differential complexes
and homology groups. 

Given two posets $P$ and $Q$, the product $P \times Q$ is a poset with partial order
\[
  (x, x') \leq (y, y') \quad \Longleftrightarrow \quad x \leq y \text{ and } x' \leq y'.
\]
If $P$ and $Q$ are graded, $P \times Q$ can be graded by setting $\dim(x,y) = \dim(x) + \dim(y)$. 
If $P$ and $Q$ are thin, then $P \times Q$ is thin as well. Indeed, an interval of length 
$2$ in the product is either constant in one factor or has the form
\[
\begin{tikzcd}[cramped, sep=small]
     & (y,y') & \\
		(x,y') \arrow[draw=none]{ru}[sloped,auto=false]{<} & & (y,x') \arrow[draw=none]{lu}[sloped,auto=false]{>} \\
		& (x,x') \arrow[draw=none]{lu}[sloped,auto=false]{>} \arrow[draw=none]{ru}[sloped,auto=false]{<} & 
\end{tikzcd}
\]
with $x \lessdot y$ and $x' \lessdot y'$. 
Given a cosheaf $\F$ on an open subset $U \subset P \times Q$, the associated
differential complex $C_\bullet(U;\F)$ can be refined into the following bicomplex
\[
\begin{tikzcd}
  & \vdots \arrow{d}{\partial_2} &\vdots \arrow{d}{\partial_2} &\vdots \arrow{d}{\partial_2} &
 \\\cdots  \arrow{r}{\ \partial_1\ } & E^0_{r+1,s+1}
 \arrow{r}{\ \partial_1\ }  \arrow{d}{\partial_2}  & 
E^0_{r,s+1}
 \arrow{r}{\ \partial_1\ }  \arrow{d}{\partial_2}  & E^0_{r-1,s+1}
 \arrow{r}{\ \partial_1\ }  \arrow{d}{\partial_2}  & \cdots
 \\ \cdots  \arrow{r}{\ \partial_1\ } & E^0_{r+1,s}
 \arrow{r}{\ \partial_1\ }  \arrow{d}{\partial_2}  & 
E^0_{r,s}
 \arrow{r}{\ \partial_1\ }  \arrow{d}{\partial_2}  & E^0_{r-1,s}
 \arrow{r}{\ \partial_1\ }  \arrow{d}{\partial_2}  & \cdots
  \\ \cdots  \arrow{r}{\ \partial_1\ } & E^0_{r+1,s-1}
 \arrow{r}{\ \partial_1\ }  \arrow{d}{\partial_2}  & 
E^0_{r,s-1}
 \arrow{r}{\ \partial_1\ }  \arrow{d}{\partial_2}  & E^0_{r-1,s-1}
 \arrow{r}{\ \partial_1\ }  \arrow{d}{\partial_2}  & \cdots
 \\ & \vdots & \vdots & \vdots &
\end{tikzcd}
\]
where 
\[
  E^0_{r,s}= \bigoplus_{\substack{(x,x')\in U\\\dim (x) = r \\ \dim(x') =
      s}} \F(x,x'),
  \quad\qquad C_q(U;\F) = \bigoplus_{r+s=q}E^0_{r,s},
\]
and the restrictions of the differentials on $\F(x,x')$ are given by 
\begin{align} \nonumber
  \partial_1 (\alpha)= \sum_{x \lessdot y  }
  \iota_{(x,x'),(y,x')}(\alpha)
  & & \text{and} & &
  \partial_2(\alpha) = \sum_{x' \lessdot y'} \iota_{(x,x'),(x,y')}(\alpha).
\end{align}
The two canonical filtrations of the bicomplex given by
\begin{align} \nonumber
  F_i = \bigoplus_{\substack{r,s \\ s \le i}} E^0_{r,s} & & \text{and} &&
  F'_i = \bigoplus_{\substack{r,s \\ r \le i}} E^0_{r,s} 
\end{align}
give rise to two spectral sequences, both converging to
$H_\bullet(U;\F)$. We refer to \cite[Section 14]{BotTu82} for 
more details. 
To fix our index notation in the spectral sequences, we describe in
Figures \ref{fig:F1} and \ref{fig:F2} the first pages of these two
spectral sequences.
In particular, a basic application of spectral sequences gives the
following.
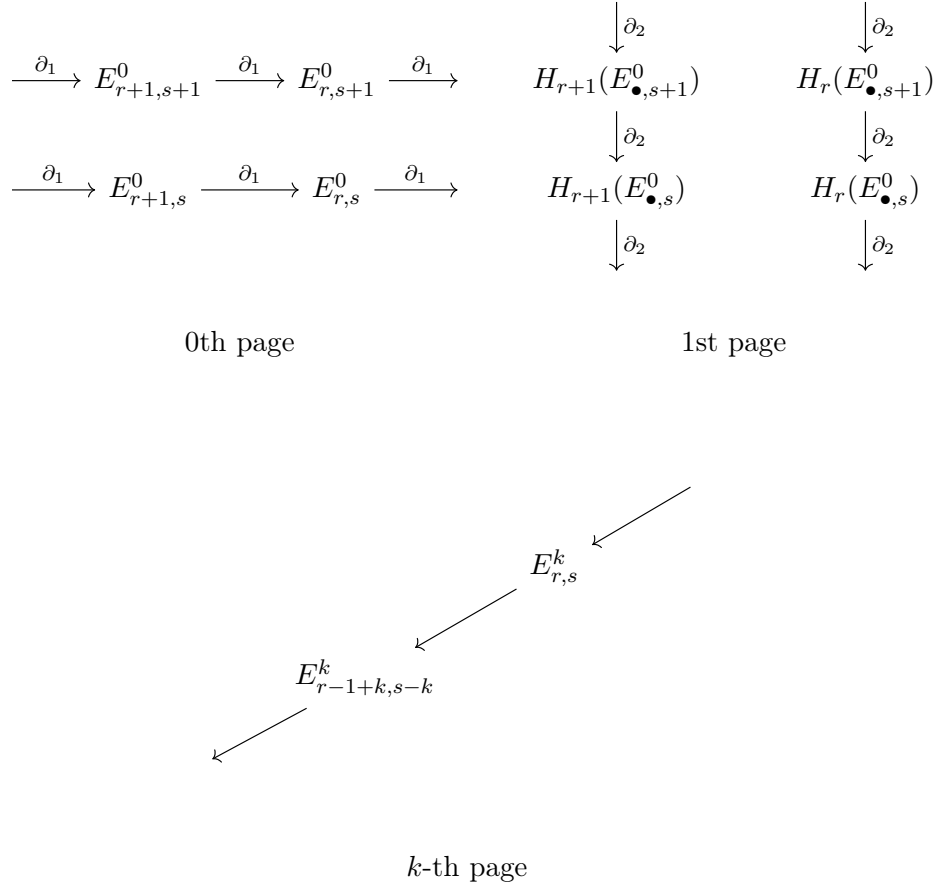
\begin{figure}[!h]
  \[
  \begin{array}{cc}
\begin{tikzcd}
 \\  \arrow{r}{\partial_1}   &  E^0_{r+1,s+1}
 \arrow{r}{\partial_1}   & 
E^0_{r,s+1}\arrow{r}{\partial_1} & \
 \\ \arrow{r}{\partial_1}   &   E^0_{r+1,s}
 \arrow{r}{\partial_1} & 
E^0_{r,s}\arrow{r}{\partial_1} &\
\end{tikzcd}
& 
\begin{tikzcd}
 \\  \arrow{d}{\partial_2} & \arrow{d}{\partial_2}  

 \\  H_{r+1}(E^0_{\bullet,s+1})
 \arrow{d}{\partial_2}   & 
H_r(E^0_{\bullet,s+1}) \arrow{d}{\partial_2}  

 \\ H_{r+1}(E^0_{\bullet,s}) \arrow{d}{\partial_2}  
  & 
 H_r(E^0_{\bullet,s}) \arrow{d}{\partial_2}
 \\ \ &\ 
\end{tikzcd}
\\\\ 0\mbox{th page} &
1\mbox{st page}
\end{array}
\]
\[
  \begin{array}{c}
\begin{tikzcd}
 \\ \ & \ &   & \phantom{E^2_{r-1+k,s-k}} 
 \\ \  & 
  & 
E^k_{r,s}  \arrow[ru,leftarrow]  &  \phantom{E^k_{r,s}} 
 \\&   E^k_{r-1+k,s-k}
 \arrow[ru,leftarrow] & 
  &\phantom{E^2_{r+1,s}} 
 \\ \phantom{E^2_{r+1,s}}  \arrow[ru,leftarrow] &  \phantom{E^2_{r+1,s}} 
\end{tikzcd}
\\\\  
k\mbox{-th page}
  \end{array}
  \]
\caption{The spectral sequence for the filtration $\displaystyle F_i = \bigoplus_{r,s \; : \; s \le i} E^0_{r,s}$  \label{fig:F1}}
  \end{figure}

\begin{figure}[!h]
  \[
  \begin{array}{cc}
 \begin{tikzcd}
 \\  \arrow{d}{\partial_2} & \arrow{d}{\partial_2}  

 \\  E^0_{r+1,s+1}
 \arrow{d}{\partial_2}   & 
E^0_{r,s+1}\arrow{d}{\partial_2}  

 \\ E^0_{r+1,s} \arrow{d}{\partial_2}  
  & 
 E^0_{r,s} \arrow{d}{\partial_2}
 \\ \ &\ 
 \end{tikzcd}
 &
 \begin{tikzcd}
 \\  \arrow{r}{\partial_1}   &  H_{s+1}(E^0_{r+1,\bullet})
 \arrow{r}{\partial_1}   & 
H_{s+1}(E^0_{r,\bullet})\arrow{r}{\partial_1} & \
 \\ \arrow{r}{\partial_1}   &  H_{s}(E^0_{r+1,\bullet})
 \arrow{r}{\partial_1} & 
H_{s}(E^0_{r,\bullet})\arrow{r}{\partial_1} &\
\end{tikzcd}
 \\\\ 0\mbox{th page} & 
 1\mbox{st page}
  \end{array}
  \]
  \[
  \begin{array}{c}
  \begin{tikzcd}
 \\ \ & \ &   & \phantom{E^2_{r-1+}} 
 \\ \  & 
  & 
E^k_{r-k,s-1+k}  \arrow[ru]  &  \phantom{E^k_{r,s}} 
 \\&   E^k_{r,s}
 \arrow[ru] & 
  &\phantom{E^2_{r+1,s}} 
 \\ \phantom{E^2_{r-1+k,s-k}}  \arrow[ru] &  \phantom{E^2_{r+1,s}} 
\end{tikzcd}
  \\\\  
  k\mbox{-th page}
  \end{array}
  \]
\caption{The spectral sequence for the filtration $\displaystyle F'_i = \bigoplus_{r,s \; : \; r \le i} E^0_{r,s}$  \label{fig:F2}}
  \end{figure}

\begin{lemma}\label{lem:p1}
  Suppose that $H_r(E^0_{\bullet,s})=0$ for all $r\ne r_0$. Then
  \[
  \forall q\in \Z,\quad H_q(U;\F)=\frac{
    \textrm{Ker}\left(\partial_2 :
    H_{r_0}(E^0_{\bullet,q-r_0})\to H_{r_0}(E^0_{\bullet,q-r_0-1})
    \right)}
          {\textrm{Im}\left(\partial_2 :
    H_{r_0}(E^0_{\bullet,q-r_0+1})\to H_{r_0}(E^0_{\bullet,q-r_0})
    \right)}.
          \]

          Analogously if $H_s(E^0_{r,\bullet})=0$ for all
          $s\ne s_0$, then
\[
  \forall q\in \Z,\quad H_q(U;\F)=\frac{
    \textrm{Ker}\left(\partial_1 :
    H_{s_0}(E^0_{q-s_0,\bullet})\to H_{s_0}(E^0_{q-s_0-1,\bullet})
    \right)}
          {\textrm{Im}\left(\partial_1 :
    H_{s_0}(E^0_{q-s_0+1,\bullet})\to H_{s_0}(E^0_{q-s_0,\bullet})
    \right)}.
          \]
\end{lemma}

\subsection{The cosheaves $\F_p^k$}
We now turn to the special cases of poset homology that are of interest to us. 
Recall that $\Delta$ is a non-singular 
lattice polytope in
$\R^n$ and $\Gamma$ is a primitive triangulation of $\Delta$.
We denote by $\Fac(\Delta)$ the face lattice of $\Delta$, graded by
the dimension. 
The triangulation $\Gamma$ is a poset with respect to inclusion and
also graded by
the dimension. 
We define the following open set of the poset
$ \Fac(\Delta) \times \Gamma^{\text{op}}$:
\[
  \Xi = \{ (F, \sigma) : \sigma \subset F\} \quad \subset \Fac(\Delta) \times \Gamma^{\text{op}}.
\]
We emphasize that we use the inverted order in the second factor $\Gamma^{\text{op}}$, that is, 
\[
(G,\tau)\leq (F,\sigma) \Longleftrightarrow 
 \sigma \subseteq \tau \subseteq G\subseteq F.
\]

\begin{remark}[For the attention of tropical experts] \label{rem:DualSubdivision}
  When  convex, the triangulation $\Gamma$ is dual to a 
  (generalized)  polyhedral subdivision $\SSS$ of
  the tropical toric variety $\mbox{Tor}_\TT(\Delta)$.
As usual $\mbox{Tor}_\TT(\Delta)$ is a disjoint union of (tropical)
toric orbits $B_F$ which are in one-to-one correspondence with  faces
$F$ of
$\Delta$.
A cell $\sigma\in \Gamma$ has a dual cell $\sigma_F^\vee\in \SSS$
contained in $B_F$ for each face $F$
of $\Delta$ containing $\sigma$. In particular the correspondence
$(F,\sigma)\in \Xi \mapsto \sigma_F^\vee\in \SSS$ establishes
an isomorphism of posets between $\Xi$ and the face poset of $\SSS$.
\end{remark}

The poset $\Xi$ is graded by
$\dim(F, \sigma) = \dim(F) -\dim(\sigma)$, which
takes  values in $\Z_{\ge 0}$.
It is well-known that $\Fac(\Delta)$ and $\Gamma$ are thin. 
Hence
$\Xi$ is thin as well.

Recall that we define cosheaves $\F_p^k$ on $\Xi$ for all $p,k = 0, \dots, n$ given by
  \[
  \F_p^k(F,\sigma)=\sum_{\substack{\tau\subset \sigma\\\dim \tau=k}}
  \bw^p \left(\tau^\perp/F^\perp \right).
  \]
  In particular $\F_p^k(F,\sigma)=0$ if $p>\dim F -k$ or $\dim \sigma<k$.
Note that $\F_p^0(F,\sigma)=\bw^p \left((\FF^n)^\vee/F^\perp \right)$
only depends on $F$, and that
for all $p$
\[
  \F_p^n(F,\sigma)\subset \F_p^{n-1}(F,\sigma)\subset \dots \subset \F_p^{1}(F,\sigma) \subset \F_p^0(F,\sigma).
\]
The next lemma gives more detailed information about this filtration. 

\begin{lemma}\label{lem:dim Fp}
  Given $1\le k\le n$, and $0\le p\le n-k$, one has
  \[
\dim \F_p^k(F,\sigma)=
     {{\dim F}\choose{p}}-\sum_{l=0}^{k-1}{{\dim
         \sigma}\choose{l}}{{\dim F-\dim\sigma}\choose{p-\dim\sigma+l}}
     \]
In particular if $p\le \dim \sigma-k$, then
  \[
  \F_p^k(F,\sigma)= \F_p^0(F,\sigma).
  \]
\end{lemma}

\begin{proof}
  Let $m$ be the dimension of the face $F$.
  Since $\sigma$ is a primitive simplex, there exists a basis
 $(e_1,\cdots,e_n)$ of
 $\Z^n$ such that $e_1,\cdots,e_{\dim\sigma}$ are
  the edges of $\sigma$ adjacent to one of its  vertices $v_0$, and
  $(e_1,\cdots,e_{m})$ is a basis of $T_\Z(F)$.
 We denote by $(e^\vee_1,\cdots,e^\vee_n)$ the basis of $(\Z^n)^\vee$ dual
 to $(e_1,\cdots,e_n)$.
We identify the quotient space $(\Z^n)^\vee/F^\perp$ with the linear
space generated by $(e_1^\vee,\cdots,e_m^\vee)$.
Given two positive integers $a,b$, we denote by
$\mathcal P(a,b)$ the set of subsets of $\{1,\cdots,a\}$ of cardinal $b$.

There is a bijection between faces of $\sigma$
of dimension $k$  containing $v_0$, and 
 elements of $\mathcal P(\dim\sigma,k)$.
The face of $\sigma$ corresponding to such subset $I$ is denoted by
$\tau_I$.
 A basis of $\tau_I^\perp/F^\perp$ is then
 $(e^\vee_i)_{i\in\{1,\cdots,m\}\setminus I}$, and a basis of
 $\bw^p \tau_I^\perp/F^\perp$ is given by
 \[
 (e^\vee_{i_1}\wedge \cdots\wedge
  e^\vee_{i_p})_{\{i_1,\cdots i_p\}\subset \{1,\cdots,m\}\setminus I}.
  \]
  Defining
  \[
  \mathfrak J=\big\{J\in \mathcal P(m,p) |\
  \exists I\in\mathcal P(\dim\sigma,k), \ J\cap I =\emptyset 
  \big\},
  \]
  we obtain that 
 the vector space
   \[
 \sum_{I\in\mathcal P(\dim\sigma,k)}
  \bw^p \left(\tau_I^\perp/F^\perp \right)
  \]
  admits
   \[
 (e^\vee_{i_1}\wedge \cdots\wedge
  e^\vee_{i_p})_{\{i_1,\cdots i_p\}\in\mathfrak J}
  \]
  as a basis.
  Hence it has dimension
  \begin{align*}
    |\mathfrak J|&= {m\choose p}-
    \left|\big\{J\in \mathcal P(m,p) |\
  \forall I\in\mathcal P(\dim\sigma,k), \ J\cap I \ne \emptyset 
  \big\}  \right|
  \\ &= {m\choose p}- \sum_{l=0}^{k-1}\big|\mathcal
  P(\dim\sigma,\dim\sigma-l)\big|
  \times  \big|\mathcal P(\dim F-\dim\sigma,p-\dim\sigma+l)\big|
  \\ &=  {{\dim F}\choose{p}}-\sum_{l=0}^{k-1}{{\dim
         \sigma}\choose{l}}{{\dim F-\dim\sigma}\choose{p-\dim\sigma+l}}.
  \end{align*}

  \medskip
  To end the proof of the lemma, it remains to observe that
     \[
 \F_p^k(F,\sigma)=\sum_{I\in\mathcal P(\dim\sigma,k)}
  \bw^p \left(\tau_I^\perp/F^\perp \right).
  \]
Indeed, let $\tau$ be a face of $\sigma$ of dimension $k$ which does
not contain the vertex $v_0$. Without loss of generality, we may
assume that $(e_1+e_2,\cdots, e_1+e_{k+1})$ is a basis of 
$T_{\FF}(\tau)$, and so that
\[
(e_1^\vee+e_2^\vee+\cdots +e_{k+1}^\vee, e_{k+1}^\vee,\cdots
e_{m}^\vee)
\]
is a basis of $\tau^\perp/F^\perp$.
In particular we see that
    \[
 \bw^p  \left(\tau^\perp/F^\perp \right)\subset \sum_{I\in\mathcal P(\dim\sigma,k)}
  \bw^p \left(\tau_I^\perp/F^\perp \right),
  \]
  which ends the proof.
\end{proof}

The definition of the cosheaves $\F^k_p$ is motivated by their relation
to the homology of linear subspaces of complex tori. Even if we will
not strictly speaking use the following proposition in the text, it
is important to keep it in mind. 
\begin{prop}[\cite{Zha13} and \cite{OT-ArrangementsHyperplanes}]\label{prop:OS}
	For any $(F,\sigma)\in\Xi$, and any integer $k\ge 0$,
  there exists an isomorphism  of $\FF$-vector spaces
  \[
  \F^k_p(F,\sigma)\stackrel{\sim}{\longrightarrow}H_p(M_{\dim
    \sigma,k}\times(\C^*)^{\dim F-\dim\sigma};\FF),
  \]
  where  $M_{n,k}$ is a generic
linear space of codimension $k$ in $(\C^*)^n$.
 \end{prop}

\begin{rem}
  Proposition \ref{prop:OS} combined with  the Lefschetz hyperplane
  section theorem  provides an alternative geometric proof of
  Lemma \ref{lem:dim Fp}.
\end{rem}

\subsection{Homology of non-singular toric varieties}
Here we relate numbers $h_q(\Gamma;\F_p^0)$ to Hodge numbers of the
toric variety $\mbox{Tor}_\C(\Delta)$
The next paragraph recasts in
convenient (for us) notations known
results about computations of homology and Chow groups of toric varieties.

Let $\Sigma\subset \R^n$ be a unimodular fan defining a non-singular toric
variety $\mbox{Tor}_\C(\Sigma)$. We can construct a differential complex
\[
\begin{array}{cccccc}
 &  E_{n,p} & & E_{n-1,p}&& E_{p,p}
 \\ &  \rotatebox[origin=c]{90}{=} & &   \rotatebox[origin=c]{90}{=}&&   \rotatebox[origin=c]{90}{=}
 \\
  0 \to &\bw^p \FF^n &\to
&\displaystyle \bigoplus_{\substack{\sigma \in \Sigma\\ \dim \sigma=1}} \bw^p \left(\FF^n/T_{\FF}(\sigma) \right) &
\to   \cdots \to &
\displaystyle  \bigoplus_{\substack{\sigma \in \Sigma\\ \dim \sigma=n-p}}  \bw^p \left(\FF^n/T_{\FF}(\sigma)
\right)\to 0
\end{array}
\]
where the differential maps a vector
$v\in \bw^p \left(\FF^n/T_{\FF}(\tau)  \right)$ to the sum of the
natural projections of $v$ to
$\bw^p \left(\FF^n/T_{\FF}(\sigma)\right)$ for all cones $\sigma$ of $\Sigma$
that contain $\tau$ as a facet, see e.g.\ \cite[Chapter 2]{Jor-HomologyCohomologyToric}.
We denote by $H_{r,p}(\Sigma)$ the $r$-th homology group of this
complex.

\begin{lemma}\label{lem:hom tor}
The Borel-Moore homology group $H^{BM}_i(\mbox{Tor}_\C(\Sigma);\FF)$ splits as a 
(non-canonical) direct sum
\[
H^{BM}_i(\mbox{Tor}_\C(\Sigma);\FF)=\bigoplus_{r+p=i}H_{r,p}(\Sigma).
\]
Furthermore, the group $H_{p,p}(\Sigma)$ is isomorphic to
$A_p(\mbox{Tor}_\C(\Sigma))\otimes \FF$, where $A_p(\mbox{Tor}_\C(\Sigma))$ is the $p$-th Chow group  
of $\mbox{Tor}_\C(\Sigma)$.
\end{lemma}

\begin{proof}
  The first part of the statement is contained in 
	\cite[Theorem 2.4.1, Proposition 2.4.5]{Jor-HomologyCohomologyToric}.
	Indeed, the complex $E_{\bullet,\bullet}$ is the first page of the
  spectral sequence for $H^{BM}_\bullet(\mbox{Tor}_\C(\Sigma);\FF)$ induced by the
	stratification of $\mbox{Tor}_\C(\Sigma)$ into torus orbits, and this
	spectral sequence degenerates at the second page.
An  isomorphism between $H_{p,p}(\Sigma)$ and
$A_p(\mbox{Tor}_\C(\Sigma))\otimes \FF$ is given by
\cite[Theorem 1]{FMSS-IntersectionTheorySpherical},
\cite[Proposition 1.1]{FulStu97} and 
the universal coefficient theorem.
\end{proof}

\begin{cor}\label{cor:hom aff}
  If $\Sigma$ is the cone generated by a basis of $\Z^n$, then
  \[
  H_{n,n}(\Sigma)=\FF\qquad\mbox{and}\qquad
  H_{r,p}(\Sigma)=0\ \mbox{if }(r,p)\ne (n,n).
  \]
\end{cor}
\begin{proof}
  In this case $\mbox{Tor}_\C(\Sigma)=\C^n$, and one computes easily
  that
    \[
  H^{BM}_{2n}(\C^n)=\FF\qquad\mbox{and}\qquad
  H^{BM}_{i}(\C^n)=0\ \mbox{if }i\ne 2n.
  \]
  Since $E_{r,p}=0$ if $r$ or $p$ is not in $\{0,\cdots,n\}$,
  the  result follows from Lemma \ref{lem:hom tor}.
\end{proof}

\begin{cor}\label{lem:eq tor}
  For all integers $p,q$ one has
  \[
  h_q(\Gamma;\F_p^0)=h^{p,q}(\mbox{Tor}_\C(\Delta)).
  \]
\end{cor}
\begin{proof}
 We compute the groups $H_q(\Gamma; \F_p^0)$ using the filtration
 $\displaystyle F'_i = \bigoplus_{r,s\;: \; r\le i}E^0_{r,s}$
 (see Figure \ref{fig:F2}).
 Each space $E^0_{r,s}$ splits into the direct sum
 \[
 E^0_{r,s}=\bigoplus_{\dim F=r}E^0_{F,s}\qquad \mbox{with}\qquad
 E^0_{F,s}=\bigoplus_{\substack{\sigma\subset F\\ \dim \sigma=-s}}\F^0_p(F,\sigma),
 \]
 and the differential complex (from Figure \ref{fig:F2} rotated by 90 degrees)
  \[
\begin{tikzcd}
 \\  \arrow{r}{\partial_2}   &  E^0_{r,s+1}
\arrow{r}{\partial_2}   &  E^0_{r,s}
\arrow{r}{\partial_2}   &  E^0_{r,s-1}
\arrow{r}{\partial_2}   & \
\end{tikzcd}
\]
splits into the direct sums of the differential complexes
  \[
\begin{tikzcd}
   \arrow{r}{\partial_2}   &  E^0_{F,s+1}
\arrow{r}{\partial_2}   &  E^0_{F,s}
\arrow{r}{\partial_2}   &  E^0_{F,s-1}
\arrow{r}{\partial_2}   & \ 
\end{tikzcd}.
\]
Since $\F^0_p(F,\sigma)=\bw^p \left((\FF^n)^\vee/F^\perp \right)$
only depends on $F$, the latter differential
complex is nothing but the standard simplicial differential complex of 
$F$ with the subdivision induced by $\Gamma$, with coefficients in
$\bw^p \left((\FF^n)^\vee/F^\perp \right)$. Since $F$ is a
contractible topological space, we obtain
\[
H_0(E^0_{F,\bullet})=\bw^p \left((\FF^n)^\vee/F^\perp \right)
\qquad\mbox{and}\qquad 
H_s(E^0_{F,\bullet})=0 \ \forall s>0,
\]
from which we deduce that
\[
H_0(E^0_{r,\bullet})=\bigoplus_{\dim F=r}\bw^p \left((\FF^n)^\vee/F^\perp \right)
\qquad\mbox{and}\qquad 
H_s(E^0_{r,\bullet})=0 \ \forall s>0.
\]
Hence by Lemma \ref{lem:p1} we obtain that the homology
$H_\bullet(\Gamma;\F^0_p)$ is the homology of the complex
  \[
\begin{tikzcd}
  0 \to \bw^p (\FF^n)^\vee \to
\displaystyle \bigoplus_{\dim F=n-1} \bw^p \left((\FF^n)^\vee/F^\perp \right)
\to   \cdots \to
\displaystyle  \bigoplus_{\dim F=p} \bw^p \left((\FF^n)^\vee/F^\perp
\right)\to 0
\end{tikzcd},
\]
where the differential maps a vector
$v\in \bw^p \left((\FF^n)^\vee/F^\perp  \right)$ to the sum of the
natural projections of $v$ to
$\bw^p \left((\FF^n)^\vee/G^\perp\right)$ for all facets $G$ of $F$.
By Lemma \ref{lem:hom tor} we obtain
\[
H_i(\mbox{Tor}_\C(\Delta);\FF)=\bigoplus_{r+p=i} H_r(\Gamma;\F^0_p).
\]

On the other hand,   the Jurkiewicz–Danilov Theorem \cite[Theorem
  10.8]{Dan78} says that
integer homology and Chow groups of $\mbox{Tor}_\C(\Delta)$ are torsion-free and
\[
b_{2p}(\mbox{Tor}_\C(\Delta))=\rk (A_p(\mbox{Tor}_\C(\Delta)))=h^{p,p}(\mbox{Tor}_\C(\Delta))
\qquad\mbox{and}\qquad
b_{2p+1}(\mbox{Tor}_\C(\Delta))=0.
\]
Hence $h_p(\Gamma;\F^0_p)=\rk (A_p(\mbox{Tor}_\C(\Delta)))= h^{p,p}(\mbox{Tor}_\C(\Delta))$ by Lemma
\ref{lem:hom tor} and we deduce that
$h_q(\Gamma;\F^0_p)=0=h^{p,q}(\mbox{Tor}_\C(\Delta))$ if $p\ne q$.
\end{proof}

\section{Homology of $\F_p^k$}\label{sec:hodge}

Here we prove Theorem \ref{thm:hodge} following the strategy
indicated in the introduction: we first prove a heredity statement
(Proposition \ref{prop:her})
in 
Section \ref{sec:heredity}, then  Poincaré duality 
(Theorem \ref{prop:dual}) in 
 Section
 \ref{sec:poincare}, and end the proof by equating tropical and
 complex $p$-characteristics in Section
\ref{sec:proof hodge}.

\subsection{Heredity}\label{sec:heredity}

Given $(F,\sigma)\in\Xi$, we denote by $F_\sigma$  the minimal face of
$\Delta$ containing $\sigma$.  
\begin{lemma} \label{lem:sigmacomplex}
  Given integers $p,k\ge 0$, 
	consider the spectral sequence from Figure \ref{fig:F1} with respect to 
	$\Xi$ and $\F = \F^k_p$. Then the first page satisfies
	\[
	  E^1_{n,s} = \bigoplus_{\dim \sigma=-s} \F^k_{p+\dim F_\sigma-n}(F_\sigma, \sigma)
\qquad\mbox{and}\qquad
 E^1_{r,s}=0 \quad \mbox{if }r\ne n.
  \]
\end{lemma}

\begin{proof}
  Note that the complex
	$E^0_{\bullet,s}$ splits into the direct sum
 \[
	 E^0_{\bullet,s}=\bigoplus_{\dim \sigma=-s}E^0_{\bullet,\sigma}\qquad \mbox{with}\qquad
	 E^0_{\bullet,\sigma}=\bigoplus_{F \supset \sigma} \F^k_p(F,\sigma).
 \]
  We denote by $H_\bullet(\sigma;\F^k_p)$ the homology (with respect to $\partial_1$) of these
  subcomplexes.
	It is hence sufficient to prove
	\[
H_n(\sigma;\F^k_p)= \F^k_{p+\dim F_\sigma-n}(F_\sigma, \sigma)
\qquad\mbox{and}\qquad
H_r(\sigma;\F^k_p)=0 \quad \mbox{if }r\ne n.
\]
for all $\sigma \in \Gamma$. To do so, let us fix suitable coordinates.  
Since $\Delta$ is a non-singular polytope, there exists a basis
$(e_1,\cdots,e_n)$ of $\Z^n$ such that
\begin{itemize}
\item $(e_1,\cdots,e_{r_0})$ is a basis of $T_{\Z}(F_\sigma)$;
\item 
  for any $F \supset F_\sigma$, a basis for $T_{\Z}(F)$ can be obtained by 
	completing
  $(e_1,\cdots,e_{r_0})$ with vectors from $\{e_{r_0+1},\cdots e_n\}$. 
\end{itemize}
The choice of the basis $(e_1,\cdots,e_n)$ induces a direct sum
decomposition
\[
\tau^\perp/F^\perp = \tau^\perp/F_\sigma^\perp\oplus F_\sigma^\perp/F^\perp
\]
for all $k$-dimensional faces $\tau$  of $\sigma$ and all faces $F$ of
$\Delta$ containing $F_\sigma$ which is compatible with the differential
$\partial_1$.
Hence we have the direct sum decomposition
\begin{align*}
  \F^k_p(F,\sigma) &=
\displaystyle \sum_{\substack{\tau\subset \sigma\\\dim \tau=k}}
  \bw^p \left(\tau^\perp/F^\perp \right)
\\ &=  \sum_{\substack{\tau\subset \sigma\\\dim \tau=k}}\left(\bigoplus_{a+b=p}
\bw^a\left( \tau^\perp/F_\sigma^\perp\right)\oplus \bw^b\left( F_\sigma^\perp/F^\perp\right)\right)
\\ &= \bigoplus_{a+b=p}\left(
\F^k_a(F_\sigma, \sigma)\oplus \bw^b\left( F_\sigma^\perp/F^\perp\right)\right),
\end{align*}
and the differential $\partial_1$ respects the $(a,b)$-direct sum decomposition.
As a consequence we obtain
\[
H_r(\sigma,\F^k_p)= \bigoplus_{a+b=p}  \F^k_a(F_\sigma, \sigma)\otimes
H_r(\D_b),
\]
where $\D_b$ is the differential complex
  \[
\begin{tikzcd}
  0 \to \bw^b F_\sigma^\perp \to
\displaystyle \bigoplus_{\substack{\dim F=n-1 \\ F_\sigma\subset F}} \bw^b \left(F_\sigma^\perp/F^\perp \right)
\to   \cdots \to
\displaystyle  \bigoplus_{\substack{\dim F=\dim F_\sigma+b \\ F_\sigma\subset F}} \bw^b \left(F_\sigma^\perp/F^\perp
\right)\to 0
\end{tikzcd}
\]
with differential mapping a vector
$v\in \bw^b \left(F_\sigma^\perp/F^\perp  \right)$ to the sum of the
natural projections of $v$ to
$\bw^b \left(F_\sigma^\perp/G^\perp\right)$ for all facets $G$ of $F$.
By Corollary \ref{cor:hom aff}  we have
\[
H_{n}(\D_{n-\dim F_\sigma})=\FF \qquad\mbox{and}\qquad
H_{r}(\D_b)=0 \quad \mbox{if }(r,b)\ne (n,n-\dim F_\sigma),
\]
which proves the claim.
\end{proof}

As a consequence we obtain that  homology of $\F^k_p$ is partially
inherited from  homology of $\F^0_p$.

\begin{prop}[Heredity]\label{prop:her}
   If $p,q$ and $k$ are such that $p+q<n-k$, then
  \[
  H_q(\Gamma;\F^k_p)= H_q(\Gamma;\F^0_p).
  \]
\end{prop}
\begin{proof}
In view of Lemma \ref{lem:p1} and Lemma \ref{lem:sigmacomplex}, it is sufficient to prove
that 
\[
\bigoplus_{\dim \sigma=-s} \F^k_{p+\dim F_\sigma-n}(F_\sigma, \sigma)
=\bigoplus_{\dim \sigma=-s} \F^0_{p+\dim F_\sigma-n}(F_\sigma, \sigma)
\]
for all $s$ such that $q = n+s \leq n-k-p$.
But for such $s$ and $\sigma$ with $\dim(\sigma) = -s$, we get
\[
  p+\dim F_\sigma-n \leq p \leq -s -k = \dim \sigma -k
\]
and hence by Lemma \ref{lem:dim Fp} the equality holds. 
\end{proof}

\subsection{Poincaré duality}\label{sec:poincare}

In this section, we prove the following theorem. 

\begin{thm}[Poincaré duality]\label{prop:dual}
  For any $k$, $p$ and $q$, we have canonical isomorphisms
	\[
	  \Hom(H_q(\Gamma;\F_p^k), \FF) \cong H_{n-k-q}(\Gamma;\F_{n-k-p}^k).
	\]
\end{thm}

We prove  Theorem \ref{prop:dual} in several steps. We start to show
in Section \ref{sec:geometric} that the poset $\Xi$ is the face poset
of a regular $CW$-complex, from which we deduce in Section \ref{sec:bary} that
the considered homology groups  are invariants under barycentric
subdivision.
We then follow the strategy from \cite{JRS-Lefschetz11Theorem}, that is, 
we consider the sheaf dual to $\F_p^k$ and its cohomology in
Section \ref{sec:coh}, define the cap product in Section \ref{sec:cap}, and we mix
all these ingredients in Section \ref{sec:PD} to eventually prove Theorem
\ref{prop:dual}. 

\subsubsection{Geometric realization of $\Xi$}\label{sec:geometric}
Recall that any regular $CW$-complex has an underlying face poset, where the
partial order is given by inclusion among faces.
\begin{proposition} \label{prop:CWstructure}
  The poset $\Xi$ is the face poset of a regular $CW$-structure for $\Delta$. 
\end{proposition}

Given a poset $P$, we denote by 
$O(P)$ the order complex of $P$. 
This is the simplicial complex whose vertices are the elements of $P$ and whose 
faces are the flags in $P$. (We use the term flag instead of chain in order to avoid conflicts
with chains in the sense of homology). 
Given a flag $X = (x_0 < \dots < x_l)$, we denote its length by 
$l(X) = l$, its first element by $\first(X) = x_0$ and its last element by $\last(X) = x_l$. 
Recall that $\Gamma$ is a triangulation of $\Delta$, and 
its barycentric subdivision is the simplicial complex $O(\Gamma)$. 
Given a flag $\sigma_\bullet \in O(\Gamma)$, we denote by $S(\sigma_\bullet)$ the corresponding simplex
in the barycentric subdivision of $\Gamma$.

\begin{definition}
  For any $(F, \sigma) \in \Xi$, we define
	\begin{gather*} 
	  O(F, \sigma) = \{\sigma_\bullet \in O(\Gamma) : \sigma
          \subset \first(\sigma_\bullet) \text{ and }
          \last(\sigma_\bullet) \subset F \},  \\
		C(F, \sigma) = \bigcup_{\sigma_\bullet \in O(F,\sigma)} S(\sigma_\bullet) \quad \subset \Delta. 
	\end{gather*}  
\end{definition}

Note that $\dim C(F,\sigma)=\dim (F,\sigma)$.
  When $\Gamma$ is convex, the subdivision
  \[
  \Delta=\bigcup_{(F,\sigma)\in\Xi} C(F,\sigma)
    \]
    is combinatorially isomorphic to  
		the dual subdivision $\SSS$ of the tropical
    toric variety $\mbox{Tor}_\mathbb T(\Delta)$ mentioned in Remark \ref{rem:DualSubdivision}. 
    \begin{example}
We illustrate in Figure
\ref{fig:trivial} some examples of cells $C(F,\sigma)$ in the case
when $\Delta$ is the standard unimodular triangle $\Delta_2$ in $\Z^2$ equipped
with the trivial subdivision.
\begin{figure}[h]
\centering
\begin{tabular}{ccccc}
  \includegraphics[height=2.8cm, angle=0]{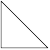}
  \put(-77,85){\tiny{$p$}}
  \put(-87,35){\tiny{$F$}}
  & \hspace{2ex} &
  \includegraphics[height=2.8cm, angle=0]{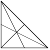}
  & \hspace{2ex} &
   \includegraphics[height=2.8cm, angle=0]{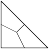}
 \\   $\Delta_2$ & &  $O(\Delta_2)$ &&   $\Delta_2=\cup_{(F,\sigma)}C(F,\sigma)$

\\ \\   \includegraphics[height=2.8cm, angle=0]{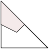}
  \put(-84,85){\tiny{$C(p,p)$}}
  \put(-107,55){\tiny{$C(F,p)$}}
  \put(-75,41){\tiny{$C(\Delta_2,p)$}}
  & \hspace{2ex} &
  \includegraphics[height=2.8cm, angle=0]{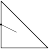}
  \put(-107,38){\tiny{$C(F,F)$}}
  \put(-75,18){\tiny{$C(\Delta_2,F)$}}
  & \hspace{2ex} &
  \includegraphics[height=2.8cm, angle=0]{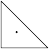}
   \put(-70,18){\tiny{$C(\Delta_2,\Delta_2)$}}
 \\  $C(\cdot,p)$&& $C(\cdot,F)$ &&  $C(\Delta_2,\Delta_2)$ 
\end{tabular}
\caption{The subdivision of $\Delta_2$
  induced by the cells $C(F,\sigma)$.}\label{fig:trivial}
\end{figure}
  \end{example}
    \begin{example}
      We depicted in Figure \ref{fig:dual} the $CW$-structure of
      $6\Delta_2$ induced by the subdivision depicted in Figure
      \ref{fig:patch}a. 
 \begin{figure}[h]
\centering
\begin{tabular}{c}
  \includegraphics[height=4cm, angle=0]{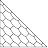}
\end{tabular}
\caption{The $CW$-structure of
      $6\Delta_2$ induced by the subdivision depicted in Figure \ref{fig:patch}a.}\label{fig:dual}
\end{figure}
   \end{example}
We claim that the cells $C(F,\sigma)$ are the cells of the regular
$CW$-structure of $\Delta$ promised in Proposition
\ref{prop:CWstructure}.  

\begin{proof}[Proof of Proposition \ref{prop:CWstructure}]
  We first note that 
	\begin{align} 
		(G, \tau) \leq (F, \sigma) & & \Longleftrightarrow && O(G,\tau) \subset O(F,\sigma) && \Longleftrightarrow && C(G,\tau) \subset C(F,\sigma),
	\end{align}
	so the order on $\Xi$ coincides with the order by inclusion of the sets $C(F,\sigma)$.
	Hence we need to check that these sets form a regular $CW$-complex. That is to say, we need to prove 
	that each cell $C(F,\sigma)$ is a ball and that its boundary $\partial C(F,\sigma)$ is 
	a union of cells.
	
	We denote by $\Omega(F, \sigma)$ the star fan of
        $F \cap  \Gamma$ with respect to the simplex $\sigma$.  
	By definition, $\Omega(F, \sigma)$ is a polyhedral fan in
        $T(F) / T(\sigma)$ whose cones are in one-to-one correspondence
	with $\sigma \subset \tau \subset F$.
        The support $C$ of
        $\Omega(F, \sigma)$ is a full-dimensional polyhedral cone in
        $T(F) / T(\sigma)$, 
	more precisely it is the inner cone of $F$ at $\sigma$. 
	We denote the barycentric subdivision of $\Omega(F, \sigma)$ by $\Omega^b(F, \sigma)$. 
	We can notice that the poset of
	cones of  $\Omega^b(F, \sigma)$ is exactly $O(F, \sigma)$.  
	Note that, when $(G, \tau) \leq (F, \sigma)$,
        the simplicial complex $C(G, \tau)$
	is simplicially homeomorphic to its image in
        $T(F) / T(\sigma)$ under the canonical projection map.  
	Moreover, up to translation, for each $\sigma_\bullet \in O(F,
        \sigma)$ the image of  
	$S(\sigma_\bullet)$ in $T(F) / T(\sigma)$ is a simplex
        obtained by cutting  
	the corresponding cone in $\Omega^b(F, \sigma)$ with a
        transversal affine hyperplane.  
	In summary $C(F, \sigma)$ is simplicially homeomorphic to the
        simplicial complex obtained by  
	intersecting $\Omega^b(F, \sigma)$ with the unit ball
        $B^{\dim(F,\sigma)} \subset T(F) / T(\sigma)$.
	But $C \cap B^{\dim(F,\sigma)}$ is a ball whose boundary is
        $(\partial C \cap B^{\dim(F,\sigma)}) \cup (C \cap \partial B^{\dim(F,\sigma)})$, 
	which under the simplicial homeomorphism corresponds to chains $\sigma_\bullet$ such that either 
	$\sigma \subsetneq \first(\sigma_\bullet)$ or
        $\last(\sigma_\bullet) \subsetneq F$. 
	Such flags belong to $O(G, \tau)$ for some $(G, \tau) < (F,
        \sigma)$, which proves the claim.  
\end{proof}

\subsubsection{The barycentric subdivision}\label{sec:bary}
The partial order on a poset $P$ induces a poset structure on the order
complex $O(P)$: given two flags $X$ and $Y$ in $O(P)$, we put $X\le Y$
if $X$ is obtained from $Y$ by removing some pieces in the flag.
In particular, we have $\last(X) \le \last(Y)$  if
$X\le Y$, and
any cosheaf  $\F$ on $P$ induces a cosheaf
on $O(P)$, still denoted by $\F$
and defined  by
\[
\F(X) = \F(\last(X)) \quad \forall X \in O(P).
\]

Given an open subset $U \subset P$, we set 
\[
  K(U) = \{X \in O(P) : \last(X) \in U\} = O(P) \setminus O(P \setminus U).
\]
Note that $O(U) \subset K(U)$ and that $K(U)$ is an open subset of $O(P) = K(P)$. 
In particular, both $O(U)$ and $K(U)$ are thin.

There is a 
chain map $\bary \colon C_\bullet(U; \F) \to C_\bullet(K(U); \F)$ defined as follows.
We 
write a chain $\alpha \in C_q(U; \F)$ 
as
\[
  \alpha = \sum_{x \in U} \alpha(x) [x]
\]
meaning that $\alpha(x)$ denotes the coefficient of $\alpha$ in the direct summand $\F(x)$.
Now, we define 
$\bary(\alpha) \in C_q(K(U); \F)$ by setting for any flag $X$ of length $q$
\[
  \bary(\alpha)(X) = \alpha(\last(X)).
\]
Here, it is understood that $\alpha(x) = 0$ if $\dim(x) \neq q$. 
It is straightforward to check that $\bary$ is a chain map. 

\begin{prop}\label{CellularSimplicial}
  Let $\F$ be a cosheaf on $\Xi$ and let $U \subset \Xi$ be an open subset. 
	Then the chain map $\bary$ from above induces isomorphisms of homology groups
	\[
	  \bary \colon H_q(U; \F) \cong H_q(K(U); \F).
	\]
\end{prop}

\begin{proof}
  We denote by $Q = \Xi \setminus U$ the closed complement of $U$.
	We first note that by definition the differential complexes
	$C_\bullet(U; \F)$ and $C_\bullet(K(U); \F)$ coincide with the
	relative complexes $C_\bullet(\Xi, Q; \F)$ and $C_\bullet(O(\Xi), O(Q); \F)$,
	respectively. 
	By Proposition \ref{prop:CWstructure}, $\Xi$ represents a regular CW complex with support $\Delta$,
	and $U$ and $Q$ represent an open and closed subsets of $\Delta$, respectively. 
	Moreover, $C_\bullet(\Xi, Q; \F)$ and $C_\bullet(O(\Xi), O(Q); \F)$ are just the cellular and simplicial
	complexes for computing the relative homology of $Q \subset \Delta$. It is well-known that 
	$\bary$ provides an isomorphism of homology groups, cf.\ \cite[Theorems 2.27 and 2.35]{Hat02}.
\end{proof}

\subsubsection{Cohomology and Mayer-Vietoris}\label{sec:coh}

Given a cosheaf $\F$ on $P$, its \emph{dual sheaf} is defined via 
\[
\G(x) 
= \Hom( \F(x), \FF)
\]
with induced restriction maps $\iota^* \colon \G(x) \to \G(y)$ for $x \leq y$. 
As for cosheaves, there is an induced sheaf on the order complex $O(P)$ given by 
$\G(X) = \G(\last(X))$.
Given an open set $U \subset P$, we denote by $C^\bullet(O(U); \G)$
the
simplicial cochain complex 
with coboundary map $\partial^*$.
The associated cohomology groups are denoted by $H^q(O(U); \G)$.

Given an inclusion of open sets $U' \subset U \subset P$, we
have chain maps
$C_\bullet(K(U); \F) \to C_\bullet(K(U'); \F)$ and 
$C^\bullet(O(U); \G) \to C^\bullet(O(U'); \G)$ 
given by restriction. We use the notation $\alpha|_{U'}$ for a restricted chain as usual. 
These restriction maps induce the Mayer-Vietoris exact sequences
\[
  0 \to C_\bullet(K(U \cup V); \F) \to C_\bullet(K(U); \F) \oplus C_\bullet(K(V); \F) \to C_\bullet(K(U \cap V); \F) \to 0
\]
and 
\[
  0 \to C^\bullet(O(U \cup V); \G) \to C^\bullet(O(U); \G) \oplus C^\bullet(O(V); \G) \to C^\bullet(O(U \cap V); \G) \to 0
\]
for two open subsets $U,V \subset P$. 
The exactness is a straightforward consequence of $K(U \cup V) = K(U) \cup K(V)$, 
$K(U \cap V) = K(U) \cap K(V)$ and $O(U \cup V) = O(U) \cup O(V)$, 
$O(U \cap V) = O(U) \cap O(V)$ (the second last equality uses that $U$ and $V$ are open).

\subsubsection{The cap product}\label{sec:cap}

We denote the dual sheaves of the cosheaves $\F_p^k$ on $\Xi$ by $\G^p_k = \Hom(\F_p^k, \FF)$.
Note that for $x \in \Xi$ we have the \emph{contraction maps}
\[
\langle \cdot  , \cdot \rangle_x \colon \G^p_k(x) \times \F_{p'}^k(x) \to \F_{p'-p}^k(x)
\]
which are induced by the usual contraction maps
\[
 \bw^p V^\vee \times \bw^{p'} V \to \bw^{p'-p} V
\]
for a vector space $V$. Note that these maps are compatible with the (co-)sheaf structures, that is, 
given $x \le y$, $\varphi \in \G^p_k(x)$ and $\alpha \in
\F_{p'}^k(y)$, we have
\[
  \langle \varphi , \iota \alpha \rangle_x = \iota \langle \iota^* \varphi , \alpha \rangle_y.
\]

Given a flag $X = (x_0 < \dots < x_l)$, we set
\begin{gather*}
  X_{\leq k} = (x_0 < \dots < x_{k}), \\
  X_{\geq k} = (x_{l-k} < \dots < x_{l}), \\
	X^{i} = (x_0 < \dots < \widehat{x_i} < \dots < x_{l}).
\end{gather*}

\begin{defi}
  The \emph{cap product} on an open set $U \subset \Xi$ is the map defined by
        \[
        \begin{array}{cccc}
          	\cap :&  C^q(O(U); \G^p_k) \times C_{q'}(K(U);
                \F_{p'}^k)&\longrightarrow & C_{q'-q}(K(U);
                \F_{p'-p}^k)
                \\ & (\varphi,\alpha)&\longmapsto& \displaystyle\sum_{\substack{X \in K(U) \\ l(X) = q' }} 
			\iota \; \langle \varphi(X_{\geq q})  , \alpha(X) \rangle_{\last(X)} \;\; [X_{\leq {q'-q}}]
        \end{array}.
       \]
\end{defi}
If $X = (x_0 < \dots < x_{q'})$ is such that $x_{q'-q} \notin U$, then
$\varphi$ is not defined on $X_{\geq q}$; 
 this is exactly the case when $X_{\leq {q'-q}} \notin K(U)$. We set 
$[X_{\leq {q'-q}}] = 0$ in such cases meaning that these terms are
 removed from the sum in the above definition. 
Given open subsets $U \subset V \subset \Xi$, it is obvious that $\cap$ commutes 
with the restriction maps from above, that is, $(\varphi \cap \alpha)|_U = \varphi|_U \cap \alpha|_U$. 
Moreover, with a calculation identical to the classical case,
see \cite[Page 240]{Hat02}, one verifies that
\[
  \partial(\varphi \cap \alpha) = \beta \cap \partial \alpha - (-1)^{q' - q} \; \partial^* \beta \cap \alpha. 
\] 
Hence we obtain induced maps on homology
\[
  \cap \colon H^q(O(U); \G^p_k) \times H_{q'}(K(U); \F_{p'}^k) \to H_{q'-q}(K(U); \F_{p'-p}^k).
\]

\subsubsection{Poincaré duality}\label{sec:PD}

For fixed value $0 \leq k \leq n$ and $m = n-k$, we define the \emph{$k$-th fundamental class} as the 
chain $ \fund^k \in C_m(\Xi; \F_m^k)$ as follows. 
Given $x = (F, \sigma) \in \Xi$ with $\dim(x) = m$, the vector space
$\sigma^\perp/F^\perp$ has dimension $m$. Hence
\[
  \F_m^k(x) = \bw^m \left(\sigma^\perp/F^\perp \right)\simeq \FF,
\]
and we denote by $\Vol(x)$  the generator in $\F_m(x)$.
We now define
\[
  \fund^k = \sum_{\substack{\sigma \in \Gamma \\ \dim(\sigma) = k}} \Vol((\Delta, \sigma)) \; [(\Delta, \sigma)] \;\;\; \in C_m(\Xi; \F_m^k). 
\]

\begin{lemma}
  We have $\partial \fund^k=0$. 
\end{lemma}

\begin{proof}
  Elements  $(F,\tau)\in C_{m-1}(\Xi;\F_m^k)$ that may contribute to 
  $\partial \fund^k$ are of two kinds.
  \begin{itemize}
  \item  $F$ is a facet of $\Delta$ and
    $\dim(\tau) = k$.
    In this case for dimensional reasons
    \[
    \F_m(F,\tau) = \bw^m \left(\tau^\perp/F^\perp \right) = 0,
    \]
    so this cell does not contribute to $\partial \fund^k$.
  \item $F=\Delta$ and $\dim(\tau) = k+1$. This case  reduces to
    the case when $k=1$ and $\Delta=\tau$ is
    the standard simplex in $\R^n$. Denoting by
  $(e_1^\vee,\cdots,e_{n}^\vee)$ the basis of $(\Z^{n})^\vee$ dual
  to the standard basis of $\Z^{n}$, the elements $\Vol(\Delta,\sigma)$ where
  $\sigma$ ranges over facets of $ \Delta_{n}$ are
  \[
  e_1^\vee,\ \cdots,\ e_{n}^\vee, \ e_1^\vee+e_2^\vee+\cdots+e_{n}^\vee.
  \]
  Hence
  \[
 \sum_{\substack{\sigma \subset \tau \\ \dim(\sigma) = k}} \Vol((\Delta, \sigma)) =0,
  \]
and this cell does not contribute to $\partial \fund^k$ neither.
  \end{itemize}
Altogether we obtain that $\partial \fund^k=0$.  
\end{proof}

We are now ready to formulate the main theorem.
By abuse of notation, we write $\fund^k$ for
$\bary(\fund^k) \in  C_m(O(\Xi); \F_m^k)$. Given
an open set $U \subset \Xi$, we write $\fund^k|_U$ for the restriction
to $C_m(K(U); \F_m^k)$. 

\begin{thm} \label{thm:PDCup}
  Let $U \subset \Xi$ be an open set. Then the cap product with the
  fundamental class 
	\begin{equation} \begin{split} 
		\cap \; \fund|_U \colon H^q(O(U); \G^p_k) &\to H_{m-q}(K(U); \F_{m-p}^k), \\
		                    \varphi &\mapsto \varphi \cap \fund^k|_U
	\end{split} \end{equation}
	is an isomorphism for all $0 \leq p,q \leq m$. 
\end{thm}

\begin{proof}
  We start by proving the statement on open stars
  $U = U(x) = \{y \in  \Xi : y > x\}$.  
	Note that $O(U)$ is contractible in the following sense:
        Consider the map  
	$\rho \colon C^q(O(U); \G^p_k) \to C^{q-1}(O(U); \G^p_k)$ given by
	\[
        \rho\varphi (X) = 
		  \begin{cases}
			  0 & \text{ if } \first(X) = x, \\
				\varphi(x < X) & \text{ otherwise}. 
			\end{cases}
	\]
        It is easy to check that
        $\partial^* \rho + \rho \partial^* = \text{Id}$
  for $q >0$.
	For $q =0$ the sum is equal to $\text{Id} - \Phi \circ \text{pr}_x$, 
	where $\text{pr}_x \colon C^0(O(U); \G^p_k) \to \G^p_k(x)$ is
        the projection 
	and $\Phi \colon \G^p_k(x)\to C^0(O(U); \G_k^p)$ is the map 
	given by $\Phi(\alpha) (y) = \iota^*_{x,y} (\alpha)$. 
        It follows that $H^q(O(U); \G^p_k) = 0$ for $q > 0$ and
        $H^0(O(U); \G^p_k) = \G^p_k(x)$.
	On the homology side, we have
        $H_{m-q}(K(U); \F_{m-p}^k) = H_{m-q}(U; \F_{m-p}^k)$
        by Proposition \ref{CellularSimplicial}.
        
	Let $x = (F, \sigma)$ and set $s = n - \dim(F)$
        and $l = \dim(x)$ and $d = \dim(\sigma)$.
 The proof that Poincaré duality holds  is now a rewriting,
 in the case of coefficients in $\FF$,
 of the proof of Poincaré duality over $\Z$ from
 \cite[Section 5]{JRS-Lefschetz11Theorem}. For the sake of brevity,
 we use notation from \cite{JRS-Lefschetz11Theorem}.
     First note 
	that  $\F_p^k(x)$ is equal to the corresponding $\Fbold_p$
        group over $\FF$ for the matroidal fan
	\[
	  V = H^k \times \R^l \times \TT^s  \;\;\; \subset \R^{d+l} \times \TT^s
	\]
        	at $(0,0,-\infty)$,
        where $H^k$ denotes the $k$-th stable
        self-intersection of the standard tropical 
	hyperplane in $\R^d$. Hence,
        $H_{m-q}(U; \F_{m-p}^k) = H_{m-q}^\BM(V; \Fbold_{m-p})$ 
	and $\G^p_k(x) = \Hom(\Fbold_{p}, \FF)$. Our cup product is
        equal to the one from 
	\cite[Definition 4.11]{JRS-Lefschetz11Theorem}: For $q= 0$,
        given
        $\varphi \in \G^p_k(x)$, the class
	$\varphi \cap \fund^k|_U \in H_{m}(U; \F_{m-p}^k)$ is represented by the chain obtained from contracting all coefficients
	of the fundamental chain with $\varphi$. 
	Hence by \cite[Corollary 5.9]{JRS-Lefschetz11Theorem} Poincaré duality holds for stars $U = U(x)$. 
	
		Note that $U(x) \cap U(y) = U(x \vee y)$ 
		if the join $x \vee y$ exists and $U(x) \cap U(y) = \emptyset$ otherwise. 
		Hence for the general case, we can copy the proof of \cite[Theorem 5.3]{JRS-Lefschetz11Theorem}, that is, 
	we proceed by induction over the number of stars that are needed to cover
	an open set $U$. The induction step is provided by the
        Mayer-Vietoris
        commutative diagram (we omit the (co-)sheaves for simplicity):
	\[
	\begin{tikzcd}
		0 \arrow[r] 
		  & C^\bullet(O(U \cup V)) \arrow[r] \arrow[d, "\cap \; \fund^k|_{U \cup V}"]	
			& C^\bullet(O(U)) \oplus C^\bullet(O(V)) \arrow[r] \arrow[d, "\cap \;\fund^k|_{U} \;\oplus \;\cap \;\fund^k|_{V}"]
			& C^\bullet(O(U \cap V)) \arrow[r] \arrow[d] \arrow[d, "\cap \; \fund^k|_{U \cap V}"]
		& 0 \\
		0 \arrow[r] 
		  & C_\bullet(K(U \cup V)) \arrow[r] 
			& C_\bullet(K(U)) \oplus C_\bullet(K(V)) \arrow[r] 
			& C_\bullet(K(U \cap V)) \arrow[r] 
		& 0
	\end{tikzcd}.
	\]
        Since the triangulation $\Gamma$ is finite, the result is proved.
\end{proof}

\begin{proof}[Proof of Theorem \ref{prop:dual}]
  We apply Theorem \ref{thm:PDCup} to $U = \Xi$ which proves
	\[
	  H^q(O(\Xi); \G^p) \cong H_{m-q}(K(\Xi); \F_{m-p}) = H_{m-q}(O(\Xi); \F_{m-p}).
	\]
	By the universal coefficient theorem, we have
	\[
	  H^q(O(\Xi); \G^p) \cong \Hom(H_q(O(\Xi);\F_p), \FF). 
	\]
  Finally, by Proposition \ref{CellularSimplicial} we have
	\[
	  H_{m-q}(O(\Xi); \F_{m-p}) \cong H_{m-q}(\Xi; \F_{m-p}), 
		  \quad \quad \Hom(H_q(O(\Xi);\F_p), \FF) = \Hom(H_q(\Xi;\F_p), \FF),
	\]
	which proves the claim. 
\end{proof}

\subsection{Euler characteristics}\label{sec:proof hodge}

So far, we proved  that
\[
h_q(\Gamma;\F_p^k)=h^{p,q}(\Delta^k) \quad \mbox{if }
p+q \ne n-k.
\]
Hence  the proof of Theorem \ref{thm:hodge} can be now reduced to
an Euler characteristic computation, which is the content of the 
 following proposition.
\begin{prop}\label{prop:eq euler}
  For any $p\ge 0$, one has
  \[
  \chi(\F^k_p)= \sum_{q\ge 0}(-1)^q h^{p,q}(\Delta^k).
  \]
\end{prop}
We prove Proposition \ref{prop:eq euler} in
this section, up to a combinatorial lemma which is
proved in its turn in Appendix \ref{sec:comb lem}.
To avoid any ambiguity, note that
our definition of  the binomial coefficient
$a\choose b$ 
 for any integers $a$ and $b$ is 
\[
{a\choose b}=\left\{
\begin{array}{ll}
  \frac{a(a-1)(a-2)\cdots (a-b+1)}{b!}
  &\mbox{if }b\ge 0
  \\ 0 & \mbox{if }b< 0
\end{array}\right.
\]
For example
\[
 {-1\choose b}=(-1)^b\quad\forall b\ge 0,
\qquad\mbox{and}\qquad
{0\choose b}=\left\{\begin{array}{ll}
1 &\mbox{if }b=0
\\ 0 & \mbox{if }b\ne 0
\end{array}\right.\quad \forall b\in\Z.
\]

Given $k,p\ge 0$, we define 
\[
e_{\Delta,k,p}= \sum_{q\ge 0}(-1)^q h^{p,q}(\Delta^k).
  \]
  The integer $e_{\Delta,k,p}$ is the $p$-characteristic of a
non-singular complete intersection $X$ of $k$
hypersurfaces in $\mbox{Tor}_\C(\Delta)$ with Newton
polytope $\Delta$. That is to say, if
 $\chi_y(\Delta^k)$ denotes the Hirzebruch genus of such 
complete intersection, one has
\[
\chi_y(\Delta^k)=\sum_{p\ge 0} e_{\Delta,k,p} y^p.
\]

Given a face $F$ of $\Delta$, we denote by
$e_{F,k,p}^o$ the $p$-characteristic of the intersection of $X$
with the torus orbit of $X$ corresponding to the face $F$. By
additivity of the $p$-characteristic, we have
\[
e_{\Delta,k,p}= \sum_{F\mbox{ face of }\Delta}e_{F,k,p}^o.
\]

Hence thanks to the  next
proposition,  the $p$-characteristic
$e_{\Delta,k,p}$ can be expressed in terms of the combinatoric of $\Delta$.
Given a face $F$ of $\Delta$,
we denote by $\nu_F(i)$ the number of simplices of dimension $i$
of $\Gamma$ contained in $F$.
Note that  by  \cite[Theorem 9.3.25]{DeLoRamSan10},
 the integer $\nu_F(i)$ only depends on $F$ and
 not on a particular choice of unimodular subdivision $\Gamma$.

\begin{prop}\label{prop:ep}
For any face $F$ of dimension $m$ of $\Delta$,  one has
  \[
  e_{F,k,p}^o= (-1)^m\left( {m\choose p}+\sum_{i\ge 0}
   \alpha_{m,k,p,i}\times \nu_F(i)\right),
  \]
  where
  \[
  \alpha_{m,k,p,i}=\sum_{l,u\ge 1}(-1)^{u}
    {{k}\choose{l}}  {{u-1}\choose{l-1}}  {{m+l}\choose{m-p+u}}  {{u-1}\choose{i}}.
  \]
\end{prop}
\begin{proof}
  Given  $l\ge 0$, we denote by
\[
\lambda_F(l)=\Card(lF \cap \Z^n)
\]
 the number of lattice points contained in the $l$-th dilate of $F$.
By \cite[Theorem 1.6]{DiRHaaNil19}, one has
\begin{align*}
  e_{F,k,p}^o&= (-1)^m{m\choose p}+ (-1)^m \sum_{l= 1}^k(-1)^l{{k}\choose{l}}
  \left(\sum_{\substack{b_1,\cdots b_l\ge 0\\ \sum_i b_i\le p}}(-1)^{\sum_i
    b_i}
      {{m+l}\choose{p-\sum_i b_i}}\lambda_F(l+\sum_i b_i)\right)
      \\ &= (-1)^m{m\choose p}+ (-1)^m\sum_{l\ge 1, \kappa\ge 0}
      (-1)^{l+\kappa}\rho(\kappa,l) {{k}\choose{l}} {{m+l}\choose{p-\kappa}}\lambda_F(l+\kappa),
\end{align*}
where $\rho(\kappa,l)$ is the number of ordered partitions of a
non-negative integer $\kappa$ into $l$ non-negative integers. Since
\[
\rho(\kappa,l)= {{\kappa+l-1}\choose{l-1}},
\]
we obtain after the change of variable $u=\kappa+l$
\begin{align*}
   e_{F,k,p}^o&= (-1)^m{m\choose p}+ (-1)^m\sum_{l\ge 1, u\ge l}
   (-1)^{u}{{k}\choose{l}}  {{u-1}\choose{l-1}}
   {{m+l}\choose{p+l-u}}\lambda_F(u).
\end{align*}
 By  \cite[Theorem 9.3.25]{DeLoRamSan10}, we have
  \[
  \lambda_F(u)=\sum_{i\ge 0}{{u-1}\choose{i}} \nu_F(i).
  \]
  Hence we obtain
  \begin{align*}
       e_{F,k,p}^o&= (-1)^m{m\choose p}+ (-1)^m\sum_{l\ge 1, u\ge
         l,i\ge 0}
  (-1)^{u}{{k}\choose{l}}  {{u-1}\choose{l-1}}
       {{m+l}\choose{p+l-u}}{{u-1}\choose{i}} \nu_F(i).
  \end{align*}
  Since the
  $ {{u-1}\choose{l-1}}=0$ if $1\le u<l$, we may sum over $u\ge 1$ in
  the above sum. That is, the proposition is proved.
\end{proof}

\begin{proof}[Proof of Proposition \ref{prop:eq euler}]
  We have to prove that
  \[
  \chi(\F^k_p)=e_{\Delta,k,p}.
  \]
By Lemma \ref{lem:dim Fp}, we have
  \begin{align*}
    \chi(\F^k_p)&= \sum_{(F,\sigma)\in\Xi}(-1)^{\dim F+\dim \sigma}\left(
  {{\dim F}\choose{p}}-\sum_{l=0}^{k-1}{{\dim
         \sigma}\choose{l}}{{\dim F-\dim\sigma}\choose{p-\dim\sigma+l}}
    \right)
    \\ & =\sum_{F\mbox{ face of }\Delta}(-1)^{\dim F}  \sum_{i\ge 0} (-1)^{i}\left(
    {{\dim F}\choose{p}}-\sum_{l=0}^{k-1}{{i}\choose{l}}{{\dim F-i}\choose{p-i+l}}\right)\nu_F(i).
  \end{align*}
Hence  it is enough to prove that for any face $F$ of $\Delta$ of
dimension $m$, one has
   \[
  e^o_{F,k,p}=(-1)^{m}  \sum_{i\ge 0} (-1)^{i}\left(
{{m} \choose{p}}-\sum_{l=0}^{k-1}{{i}\choose{l}}{{m-i}\choose{p-i+l}}\right)\nu_F(i).
\]
Since $ \sum_{i\ge 0} (-1)^{i}\nu_F(i)=1$, we have
\begin{multline*}
 (-1)^{m}\sum_{i\ge 0} (-1)^{i}\left(
{{m}
  \choose{p}}-\sum_{l=0}^{k-1}{{i}\choose{l}}{{m-i}\choose{p-i+l}}\right)\nu_F(i) \\
	=
(-1)^{m}\left(
{{m}
  \choose{p}}-\sum_{i\ge 0} (-1)^{i}\sum_{l=0}^{k-1}{{i}\choose{l}}{{m-i}\choose{p-i+l}}\nu_F(i)
\right).
\end{multline*}
Hence by Proposition \ref{prop:ep}, we are reduced to prove that
\[
\forall m,k,p,i\ge 0,\quad
(-1)^{i+1}\sum_{l=0}^{k-1}{{i}\choose{l}}{{m-i}\choose{p-i+l}}=
\sum_{l,u\ge 1}(-1)^{u}
    {{k}\choose{l}}  {{u-1}\choose{l-1}}  {{m+l}\choose{m-p+u}}  {{u-1}\choose{i}},
    \]
    which is done in Lemma \ref{lem:identity} below.
\end{proof}

\section{Real phase structures and $T$-manifolds}\label{sec:realstructures}

\subsection{Real phase structures}
Recall that, for a face $\sigma$ of $\Gamma$, we use the notation
\[
  \sigma^\vee = (\FF^n)^\vee / \sigma^\perp = T_{\FF}(\sigma)^\vee,
  \]
  and that $\pi_{\sigma} \colon (\FF^n)^\vee \to \sigma^\vee$ and
  $\pi_{\sigma, \tau} \colon \tau^\vee \to \sigma^\vee$
  denote the canonical projection map, the latter being defined only
  when 
  $\sigma \subset \tau$.
The translation of the real phase structures studied in \cite{RenRauSha21}
to our polytope setting is as follows.

\begin{definition} \label{def:real}
  A \emph{real phase structure} $\EEE$ on the $k$-skeleton of  $\Gamma$ consists of a choice of a point
	\[
	  \EEE(\sigma) \in \sigma^\vee
	\]
	for every $\sigma \in \Gamma$ of dimension $k$ such that, for every $\tau \in \Gamma$ of dimension $k+1$ and  $s \in \tau^\vee$,
	the cardinal $n_\EEE(\tau, s)$  of the set 
	\[
	  \{\sigma \lessdot \tau : \pi_{ \sigma,\tau}(s) = \EEE(\sigma)\}
	\]
	is even.
\end{definition}

\begin{remark}
  The parity condition can be restated as follows:
  for any $\sigma \lessdot \tau$, the preimage
	$\pi^{-1}_{\sigma,\tau}(\EEE(\sigma))$ is an $\FF$-affine line in $\tau^\vee$. 
	The condition states that the union of these lines covers any $s \in \tau^\vee$ an even number of times 
	(possibly $0$ times).
  It is shown in \cite[Lemma 2.8]{RenRauSha21} that $n_\EEE(\tau, s)$
  takes only the values $0$ and $2$, and that
	 the union of lines form a so-called necklace of lines. 
\end{remark}

\begin{example} \label{ex:rps}
  Let us have a look at the special cases $k= 0, 1,$ and $ n$. 
  \begin{enumerate}
		\item 
		When $k=0$, the space $\sigma^\vee = (\FF^n)^\vee /
                \sigma^\perp = \{0\}$ is trivial for any vertex
                $\sigma$ of $\Gamma$. Then the unique choice
		$\EEE_0 (\sigma)= 0$ defines a real phase structure on the
                $0$-skeleton of $\Gamma$
                since
                $n_{\EEE_0}(\tau, s) = 2$ for any edge $\tau$ of
                $\Gamma$ and any  $s \in \tau^\vee$.

  \item
    In the other extreme case $k= n$, the
                         condition from Definition \ref{def:real} is
                         empty. In particular any
                         choice of 
 a point $\EEE(\sigma) \in (\FF^n)^\vee$
 for every full-dimensional $\sigma$ defines a
 real phase structure on the
                $n$-skeleton of $\Gamma$.

		\item 
		When $k=1$, a real phase structure can be
                alternatively described by a sign distribution 
		  $\epsilon \colon \Delta \cap \Z^n \to \FF$
		up to flipping all signs, that is, replacing $\epsilon$ by $\epsilon+1$. 
		The relation is as follows: given an edge $\sigma \in
                \Gamma$ with end points  
		$v$ and $w$, we identify $\sigma^\vee $ with $ \FF$
                and require that
		\[
		  \epsilon(v) + \epsilon(w) + \EEE(\sigma) = 1. 
		\]
		In laymen terms: if $\EEE(\sigma) = 0$, we require $v$ and $w$ to have different signs, and conversely
		if $\EEE(\sigma) \neq 0$, then $v$ and $w$ are asked to have the same sign. 
		The condition that $n_\EEE(\tau, 0)$ is even for any triangle $\tau$ ensures that
		the above equation is consistent when going around the three edges in $\tau$,
		and hence such $\epsilon$ is uniquely defined after fixing some initial sign. 
		Vice versa, given a sign distribution the above
                equation defines a real phase structure  on edges of
                $\Gamma$.            

                Note that 
                up to coordinate change,  the union of the three lines
                $\pi^{-1}_{\sigma,\tau}(\EEE(\sigma))$ for $\sigma\lessdot\tau$     in
                $\tau^\vee$
		is one of the two arrangments depicted in Figure \ref{necklaces}.
		The case on the right hand side occurs exactly when
                all vertices of $\tau$ have equal sign, 
		the left hand side picture occuring when  one vertex
                of $\tau$ has different sign than the other two. 
		\begin{figure}[h]
		  \input{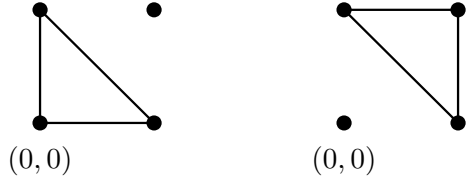}
			\label{necklaces}
			\caption{\centering The two types of necklaces
                          for real phase structures on edges.} 
		\end{figure}
	\end{enumerate}
\end{example}

\begin{exa}\label{ex:patch dte} 
As an example of a real phase structure for values $1< k< n$, we consider
$\Delta=\Gamma=\Delta_3$ the standard simplex in $\R^3$. Denote by
$\sigma_x$, $\sigma_y$, $\sigma_z$, and
$\sigma_u$
the 2-dimensional face of $\Delta_3$ contained in the
 hyperplane with equation $x=0$, $y=0$, $z=0$, and $x+y+z=1$,
respectively. 
Identifying $\FF^3$ and $(\FF^3)^\vee$ thanks to the canonical basis
of $\R^3$, we have 
\[
\sigma_x^\vee=\FF^3/\FF(1,0,0),\quad
\sigma_y^\vee=\FF^3/\FF( 0,1,0),\quad
\sigma_z^\vee=\FF^3/ \FF(0,0,1),\quad
\sigma_u^\vee=\FF^3/\FF(1,1,1).
\]
 Hence an example of a real phase structure on the 2-skeleton of
 $\Delta_3$ is given by
 \[
 \EE(\sigma_x)=\pi_{x}(0,0,0)=\pi_{x}(1,0,0),\quad
  \EE(\sigma_y)=\pi_{y}(1,0,0)=\pi_{y}(1,1,0),
 \]
  \[
 \EE(\sigma_z)=\pi_{z}(0,0,0)=\pi_{z}(0,0,1),\quad
  \EE(\sigma_u)=\pi_{u}(1,1,0)=\pi_{u}(0,0,1).
 \]
where $\pi_a$ is a shorthand for the projection $\pi_{\Delta_3,\sigma_a}$.
\end{exa}

\subsection{$T$-manifolds}
Here we define the $T$-manifold  $X_{\Gamma, \EEE}$ associated to a
real phase structure $\EEE$ on the $k$-skeleton of the barycentric
subdivision $\Gamma$. We start by describing the case $k=0$, for which
there is a unique real phase structure $\EEE_0$ by Example
\ref{ex:rps}$(1)$, and which contains all $T$-manifolds with a given
non-singular polytope $\Delta$.

To construct the $T$-manifold  $X_{\Gamma, \EEE_0}$, we glue
$2^n$ disjoint
copies of $\Delta$, labelled by $s \in (\FF^n)^\vee$ and denoted $\Delta(s)$,
by identifying faces $F \subset \Delta(s)$ and
$F \subset  \Delta(t)$ if and only if 
$s+t \in F^\perp$.
This is a classical construction in toric geometry, and
$X_{\Gamma,  \EEE_0}$ is homeomorphic to the real part
$\mbox{Tor}_\R(\Delta)$ of $\mbox{Tor}_\C(\Delta)$ by \cite[Theorem 5.4]{GKZ}.
The $CW$-structure on $\Delta$ given by the cells $C(F, \sigma)$ from
 Section \ref{sec:geometric} induces
 a $CW$-structure on $X_{\Gamma, \EEE_0}$. Given
 $s \in (\FF^n)^\vee$,
        we denote by 
	$C(F, \sigma, s)$ the copy of $C(F, \sigma)$
        in $X_{\Gamma,  \EEE_0}$ which is the image of the copy of $C(F,\sigma)$
        in $\Delta(s)$.
	Note that $C(G, \tau, t) \subset C(F, \sigma, s)$ if and only if
	$(G,\tau) \le (F, \sigma)$ and $s+t\in G^\perp$. In particular
		$C(F, \sigma, s)=C(F, \sigma, t)$ 
        if and only if $s+t \in F^\perp$, that is to say
        $C(F, \sigma, s)$ only depends on 
the class of $s$ in $F^\vee=(\FF^n)^\vee/F^\perp$.

\begin{definition} \label{def:Tamnifold}
  Let $\EEE$ be a real phase structure on the $k$-skeleton of
  $\Gamma$.
  For $(F,\sigma) \in \Xi$, we set
	\[
	\EEE(F, \sigma) = \{ s \in F^\vee
        : \pi_\tau(s) = \EEE(\tau) \text{ for some } \tau \subset
        \sigma\mbox{ with }\dim(\tau) = k\},
	\]
	and 
	\[
	  \Xi(\EEE) = \left\{ (F, \sigma, s) : (F,\sigma) \in \Xi, \ s \in \EEE(F,\sigma)\right\}.
	\]
	The \emph{$T$-manifold} $X_{\Gamma, \EEE}$ is the space 
	\[
	  X_{\Gamma, \EEE} = \bigcup_{(F, \sigma, s) \in \Xi(\EEE)}
          C(F, \sigma, s) \quad \subset X_{\Gamma, \EEE_0}.
	\]
\end{definition}

\begin{remark} \label{rem:SimplicialTManifold}
  Given that the $CW$-structure on $\Delta$ is constructed as a coarsening of 
	the barycentric subdivision of $\Gamma$, the latter 
	induces in its turn a simplical structure on $X_{\Gamma, \EEE}$ refining
        the given $CW$-structure.  
	The simplices of this subdivision are labelled by elements 
	\[
	  (\sigma_1  \subsetneq \dots   \subsetneq\sigma_{l}, s)
	\]
	where $s \in F_{\sigma_l}^\vee$
        is such that there exists 
        	$\sigma \subset \sigma_1$ with
	$\dim(\sigma) = k$ and $\pi_\sigma(s) = \EEE(\sigma)$. (Recall
        that $F_\sigma$ denotes the minimal face of $\Delta$
        containing $\sigma$.)
\end{remark}

Let us describe the $T$-manifolds associated to the preceding examples
of real phase structures. 
\begin{example} 
When $k=0$ the construction from Definition \ref{def:Tamnifold}
recovers the space $X_{\Gamma, \EEE_0}$, hence the notation is consistent. 
\end{example}

\begin{example}
  When $k=n$, each full-dimensional simplex $\sigma$ of $\Gamma$
  contributes a single point 
  \[
  C(\Delta, \sigma, \EEE(\sigma))\in \Delta(\EEE(\sigma)) \subset X_{\Gamma, \EEE}.
  \]
			In particular $X_{\Gamma, \EEE}$ is a union of $\Vol(\Delta)$ points, where $\Vol(\Delta)$
			denotes the lattice volume of $\Delta$. 

\end{example}

\begin{example}
  The $T$-curve in $\RP^2$ corresponding to the patchworking
   from Figure \ref{fig:patch} is depicted in Figure
  \ref{fig:patch2}. 
\begin{figure}[h]
\centering
\begin{tabular}{c}
  \includegraphics[height=6cm, angle=0]{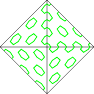}
\end{tabular}
\caption{The $T$-sextic from the patchworking of Figure \ref{fig:patch}.}\label{fig:patch2}
\end{figure}
More generally when $k=1$
it is well known,  see for example
                        \cite[Section 2]{IteVir06}, that
the $T$-hypersurface $X_{\Gamma, \varepsilon}$ constructed in Section
\ref{sec:patch1} is  ambient isotopic in $X_{\Gamma, \EEE_0}$ 
			to the $T$-hypersurface $X_{\Gamma, \EEE}$
                        from Definition \ref{def:Tamnifold}, via an
                        isotopy which 
                        fixes the simplices of $\Gamma$. 
 Given an edge $\sigma$ with endpoints
 $v$ and $w$, and $s \in (\FF^n)^\vee$,
it follows from Example \ref{ex:rps}
 that $\pi_\sigma(s) = \EEE(\sigma)$  
 if and only if $v$ and $w$ have different signs
                        in the extended sign distribution with respect
                        to the orthant $s$.  
			It is
                        hence sufficient to compare the two constructions
                        for a single orthant, say for $s=0$.  
			Using the simplicial descriptions from Remark
                        \ref{rem:SimplicialTManifold}, we note that 
			$X_{\Gamma,\EEE} \cap \Delta(0) \subset
                        \Delta$ is the simplicial complex consisting 
			of the simplices in the barycentric subdivision of $\Gamma$ labelled
			by flags
			\[
			  \sigma_1 \subsetneq \dots  \subsetneq \sigma_l
			\]
			subject to the conditions
                        that $\dim(\sigma_1)\ge 1$ and
                        $0 \in \EEE(\Delta, \sigma_1)$. Let us restrict 
			our attention to the cells contained in a fixed
                        simplex $\sigma \in \Gamma$,
			that is we require $\sigma_k \subset
                        \sigma$. It is now straightforward to check  
			that the union of such cells is ambient
                        isotopic in $\sigma$ to the convex hull of 
			the midpoints of edges in $\sigma$ with
                        different signs.
\end{example}
\begin{example}
 The $T$-line in $\RP^3$ corresponding to the real phase structure from Example
 \ref{ex:patch dte} is depicted in Figure \ref{fig:patch dte}. 
\end{example}

For general $k$, it is clear from the definition that $X_{\Gamma, \EEE}$
is a pure-dimensional simplicial complex of dimension $n-k$.

\begin{proposition} 
  The $T$-manifold $X_{\Gamma, \EEE}$ is $PL$-smooth.
\end{proposition}
\begin{proof}
  We show that the links of all
vertices
  of $X_{\Gamma, \EEE}$
	with respect to the simplicial structure from Remark \ref{rem:SimplicialTManifold}
	are $PL$-spheres, see \cite[Section 2.21]{RS-IntroductionPiecewiseLinear}.
        Let $\sigma$ be a simplex of $\Gamma$ and $s\in F_\sigma^\vee$
        be such
      that  there exists 
	$\tau \subset \sigma$ with 
	$\dim(\tau) = k$ and $\pi_{\tau}(s) = \EEE(\tau)$.
      Then the link of  $X_{\Gamma, \EEE}$ at $(\sigma,s)$
      is the simplicial complex given by
      completions
      \[
      (\sigma_1 \subsetneq \dots  \subsetneq \sigma_\lambda = \sigma 
      \subsetneq  \dots  \subsetneq  \sigma_l, s')
      \]
	with $s = \pi_{\sigma}(s')$. 
	By standard arguments, this complex is the join of the complexes 
	corresponding to completions "before" and "after" $\sigma$,
        respectively, 
	hence it is sufficient to prove that those latter are $PL$-spheres. 
	
	The complex of completions "before" $\sigma$ is labelled by flags
	$\sigma_1 \subsetneq \cdots \subsetneq \sigma_\lambda = \sigma$
	such that there exists 
	$\tau \subset \sigma_1$ with 
	$\dim(\tau) = k$ and $\pi_{\tau}(s) = \EEE(\tau)$.
	Then the statement is proven in \cite[Proposition 2.21]{RenRauSha21}
	in its dual version, that is, for (subfans of) the normal fan
	of $\sigma$ in the real vector space dual to the linear span of $F_{\sigma}$.
	
	The complex of completions "after" $\sigma$ is labelled by elements
	$(\sigma \subsetneq \sigma_1 \subsetneq \dots  \subsetneq \sigma_l, s')$ such that $s = \pi_{\sigma}(s')$. 
	After a coordinate change we can assume that 
	the inner  cone of $\Delta$ at $\sigma$, which is also the 
	support of the star fan of $\Gamma$ at $\sigma$, is
	equal to $\R^{n_1} \times \R_{\geq 0}^{n_2}$, 
	with $n_1 = \dim F_{\sigma} - \dim \sigma$ and $n_2 = n - \dim F_{\sigma}$.
	The $2^{n_2}$ choices for $s'$ correspond to the 
	$2^{n_2}$ orthants of $\R^{n_2}$
	and gluing these copies 
	produces the fan in $\R^{n_1 + n_2}$ 
	obtained by reflecting the fan in $\R^{n_1} \times \R_{\geq 0}^{n_2}$
	to the other orthants. 
	Hence   the complex in question is the link of the
  barycentric subdivision 
	of this complete fan, and is hence a $PL$-sphere. 
	\end{proof}

\subsection{The sign cosheaf}

We now define the analogue in our setting of the sign cosheaf for tropical varieties with real phase structures
defined in \cite{RenSha18} and \cite{RenRauSha22}. 
Note that given $(G, \tau) \leq (F,\sigma)$,  the image of $\EEE(F,\sigma)$ under the canonical projection $\pi_{G,F}:F^\vee \to G^\vee$
is contained in $\EEE(G,\tau)$.

\begin{definition}
  Let $\EEE$ be a real phase structure on the $k$-skeleton of $\Gamma$. The \emph{sign cosheaf} on $\Xi$ 
	is given by 
	\[
	  \SSS(F,\sigma) = \FF^{\EEE(F,\sigma)},
	\]
	with the cosheaf map for $(G, \tau) \leq (F,\sigma)$ that sends 
	the basis element $e_s$, with $s \in \EEE(F,\sigma)$, to the
        basis element $e_{\pi_{G,F}(s)}$. 
\end{definition}

Following \cite{RenSha18,RenRauSha22}, we explain in the next two
propositions how   
homology groups of $\SSS$ are related to both  homology groups of
$X_{\Gamma,\EE}$ and  homology groups of
$\F_{p}^k$.

\begin{proposition} \label{prop:signcosheafhomology}
  For every integer $q$, the  groups $H_q(\Xi; \SSS)$ and
  $H_q(X_{\Gamma,\EEE}; \FF)$ are isomorphic.
\end{proposition}

\begin{proof}
  We have a canonical identification between the chain complex of
  $\FF$ on the poset $\Xi(\EEE)$ and the chain complex of $\SSS$ on $\Xi$:
	\[
	  C_\bullet(\Xi(\EEE); \FF) = \bigoplus_{(F, \sigma, s) \in \Xi(\EEE)} \FF = \bigoplus_{(F, \sigma) \in \Xi} \FF^{\EEE(F, \sigma)} =  C_\bullet(\Xi; \SSS).
	\]
        To check that the two differentials agree, one simply has to observe
        that given $(G, \tau) \leq (F,\sigma)$, the only element 
        of the form $(G, \tau, t)$ in the boundary of
        the cell  $(F, \sigma,s)$
is $(G, \tau, \pi_{G,F}(s))$. Hence the two  differential complexes
are canonically isomorphic, and we have the group isomorphisms
	\[
	  H_q(X_{\Gamma, \EEE}; \FF) \cong H_q(\Xi(\EEE); \FF) \cong H_q(\Xi; \SSS)
	\]
 as announced. 
 \end{proof}

\begin{proposition}[\cite{RenSha18,RenRauSha22}] \label{prop:filtration}
  Let $\EEE$ be a real phase structure on the $k$-skeleton of  $\Gamma$.
  There exists a filtration of cosheaves on $\Xi$
	\[
	  0 = \SSS_{n+1} \subset \SSS_n \subset \dots \subset \SSS_0 = \SSS
	\]
	such that 
	\[
	  \SSS_p / \SSS_{p+1} \cong \F_p^k\qquad \forall 0 \leq p \leq n.
	\]
\end{proposition}

\begin{proof}
  We can use the exact same filtration and arguments as in 
	\cite[Definition 4.5 and Lemma 4.8]{RenSha18} (for $k=1$)
	and \cite[Definition 2.27 and Proposition 2.29]{RenRauSha22} (for arbitrary $k$).
\end{proof}

\begin{proof}[Proof of Theorem \ref{thm:main}]
  To finish the proof, we use the same spectral sequence argument as in 
  \cite[Theorem 1.5]{RenSha18}.
  The spectral sequence associated to the filtration from Proposition \ref{prop:filtration} has $0$-th page
	$E^0_{p,q} = C_q(\Xi; \SSS_p / \SSS_{p+1})$. Hence by Proposition \ref{prop:filtration} we have
	$E^1_{p,q} = H_q(\Xi; \F_p^k)$. Therefore $\dim(E^1_{p,q}) = h_{q}(\Gamma;\F_p^k) =h^{p,q}(\Delta^k)$
	by Theorem \ref{thm:hodge}. On the other hand, by Proposition \ref{prop:signcosheafhomology}
	we have $b_q(X_{\Gamma, \EEE}, \FF) = \dim(H_q(\Xi; \SSS)) = \sum_p \dim(E^\infty_{p,q})$.
Theorem \ref{thm:main} then follows from the inequality $\dim(E^\infty_{p,q}) \leq \dim(E^1_{p,q})$.
\end{proof}

\begin{proof}[Proof of Theorem \ref{thm:chi}]
  Using the same arguments as in the proof of Theorem \ref{thm:main}, we get 
	\[
	  \chi(X_{\Delta,\EEE}) = \sum_{p,q} (-1)^q
          \dim(E^\infty_{p,q}) = \sum_{p,q} (-1)^q \dim(E^1_{p,q}) =
          \sum_{p,q} (-1)^q h^{p,q}(\Delta^k) = \sigma(\Delta^k),
	  \]
          where the last equality follows from Hodge index theorem, see for
  example \cite[Theorem 6.33]{Voi02}.
\end{proof}

\appendix
\section{A combinatorial lemma}\label{sec:comb lem}
Given  integers $m,k,p,i$, we define the following numbers
\[
\alpha_{m,k,p,i}=\sum_{l,u\ge 1}(-1)^{u}
    {{k}\choose{l}}  {{u-1}\choose{l-1}}  {{m+l}\choose{m-p+u}}  {{u-1}\choose{i}},
\]
and
\[
\beta_{m,k,p,i}=(-1)^{i+1}\sum_{l=0}^{k-1}{{i}\choose{l}}{{m-i}\choose{p-i+l}}
\]

\begin{lemma}\label{lem:identity} 
  For any integers $p\in\Z$ and $m,k,i\ge 0$, one has
  \[
  \alpha_{m,k,p,i}=\beta_{m,k,p,i}.
  \]
\end{lemma}
\begin{proof}
  First note that by Pascal's relation, we have
  \[
  \alpha_{m,k,p,i}=  \alpha_{m-1,k,p,i}+\alpha_{m-1,k,p-1,i}
  \]
  and
    \[
  \beta_{m,k,p,i}=  \beta_{m-1,k,p,i}+\beta_{m-1,k,p-1,i}.
  \]
  Hence it is enough to prove the $m=0$ case:
  \[
    \alpha_{0,k,p,i}=\beta_{0,k,p,i}.
    \]
To do so, we prove that both sequences of numbers satisfy the
same recursion, and have the same initial values.

    \medskip
\noindent {\bf Step 1: rewriting of $ \alpha_{0,k,p,i}$.}
    We have
    \begin{align*}
      \alpha_{0,k,p,i}&=
      \sum_{l,u\ge 1}(-1)^{u}
    {{k}\choose{l}}  {{u-1}\choose{l-1}}  {{l}\choose{u-p}}  {{u-1}\choose{i}}.
    \end{align*}
    Since
    \[
    {{k}\choose{l}}  {{l}\choose{u-p}}= {{k}\choose{u-p}} {{k+p-u}\choose{l+p-u}},
    \]
    we obtain
 \begin{align*}
 \alpha_{0,k,p,i}&=
 \sum_{u\ge 1}(-1)^{u}  {{k}\choose{u-p}}{{u-1}\choose{i}}\left(
 \sum_{l\ge 1} {{u-1}\choose{l-1}}  {{k+p-u}\choose{l+p-u}}\right).
 \end{align*}
 We deduce from {\cite[Identity 5.23]{GKP94}} that for any $u\ge 1$,
 \[
 \sum_{l\ge 1} {{u-1}\choose{l-1}}  {{k+p-u}\choose{l+p-u}}=
     {{k+p-1}\choose{p}},
     \]
     which gives
 \begin{align*}
 \alpha_{0,k,p,i}&=
 \sum_{u\ge 1}(-1)^{u}  {{k}\choose{u-p}}{{u-1}\choose{i}} {{k+p-1}\choose{p}}.
 \end{align*}
 Since $k\ge 0$, we deduce from {\cite[Identity 5.24]{GKP94}} that
 \[
 \sum_{u\ge 1}(-1)^{u}  {{k}\choose{u-p}}{{u-1}\choose{i}} =\left\{
 \begin{array}{ll}
   (-1)^{k+p}  {{p-1}\choose{i-k}} &\mbox{if }p\ge 1
   \\
   \\   (-1)^{k}  {{-1}\choose{i-k}} -
        {{-1}\choose{i}}{{k-1}\choose{0}}
        = (-1)^{k}  {{-1}\choose{i-k}} - (-1)^{i}   &\mbox{if }p=0
 \end{array}
 \right..
 \]
 We finally obtain
  \begin{align*}
 \alpha_{0,k,p,i}&=\left\{
 \begin{array}{ll}
   (-1)^{k+p}  {{p-1}\choose{i-k}} {{k+p-1}\choose{p}} &\mbox{if }p\ne 0
   \\
   \\     0 &\mbox{if }p= 0 \mbox{ and } i\ge k
   \\
  \\    (-1)^{i+1}  &\mbox{if }p= 0 \mbox{ and } i\le k-1
 \end{array}
 \right..
 \end{align*}

  \medskip
  \noindent {\bf Step 2: recursion for $ \alpha_{0,k,p,i}$.}
  Applying twice Pascal's relation, we have for $p\ge 2$ and $k,i\ge 1$
  \begin{align*}
    \alpha_{0,k,p,i}&=  (-1)^{k+p}  {{p-1}\choose{i-k}}
          {{k+p-2}\choose{p}}
          +  (-1)^{k+p}  {{p-1}\choose{i-k}} {{k+p-2}\choose{p-1}}
\\ &= -   \alpha_{0,k-1,p,i-1} +    (-1)^{k+p}  {{p-2}\choose{i-k}}
   {{k+p-2}\choose{p-1}} +    (-1)^{k+p}  {{p-2}\choose{i-k-1}}
   {{k+p-2}\choose{p-1}}
   \\&= -   \alpha_{0,k-1,p,i-1} -
   \alpha_{0,k,p-1,i}-\alpha_{0,k,p-1,i-1}.
  \end{align*}
  One sees easily from Step 1 that this recursion actually holds for
 all $p\in \Z$. 
Hence all numbers $ \alpha_{0,k,p,i}$ can be computed from the
following  initial values:
\[
\alpha_{0,k,p,i}=0 \mbox{ for }p<0,\quad
\alpha_{0,0,p,i}=0 ,\quad
\alpha_{0,k,p,0}=\left\{
 \begin{array}{ll}
  -1 &\mbox{if }p= 0\mbox{ and }k\ge 1
   \\     0 &\mbox{otherwise }
 \end{array}
 \right..
\]

  \medskip
  \noindent {\bf Step 3: recursion for $ \beta_{0,k,p,i}$.}
  We have
  \begin{align*}
    \beta_{0,k,p,i}&=
    (-1)^{i+1}\sum_{l=0}^{k-1}{{i}\choose{l}}{{-i}\choose{p-i+l}}
    \\ &= (-1)^{p+1}\sum_{l=0}^{k-1}(-1)^{l}{{i}\choose{l}}{{p-1+l}\choose{p-i+l}}.
  \end{align*}
  In particular we see that
  \[
    \beta_{0,k,p,i}=0  \qquad\mbox{if } p<0.
    \]
    Applying again twice  Pascal's rule, we obtain
  \begin{align*}
   \beta_{0,k,p,i}&=
   (-1)^{p+1}\sum_{l=0}^{k-1}(-1)^{l}{{i}\choose{l}}{{p-2+l}\choose{p-1-i+l}}+
   (-1)^{p+1}\sum_{l=0}^{k-1}(-1)^{l}{{i}\choose{l}}{{p-2+l}\choose{p-i+l}}
   \\
   \\ &= -  \beta_{0,k,p-1,i} +
   (-1)^{p+1}\sum_{l=0}^{k-1}(-1)^{l}{{i-1}\choose{l-1}}{{p-2+l}\choose{p-i+l}}+
   (-1)^{p+1}\sum_{l=0}^{k-1}(-1)^{l}{{i-1}\choose{l}}{{p-2+l}\choose{p-i+l}}
   \\
   \\ &= -  \beta_{0,k,p-1,i} - \beta_{0,k-1,p,i-1} -\beta_{0,k,p-1,i-1}.
  \end{align*}
In particular we see that both sequences $\alpha_{0,k,p,i}$ and $\beta_{0,k,p,i}$
 satisfy the
same recursion. Furthermore we have the following initial values 
\[
\beta_{0,k,p,i}=0 \mbox{ for }p<0,\quad
\beta_{0,0,p,i}=0 ,\quad
\beta_{0,k,p,0}=\left\{
 \begin{array}{ll}
  -1 &\mbox{if }p= 0\mbox{ and }k\ge 1
   \\     0 &\mbox{otherwise }
 \end{array}
 \right..
\]
Since both sequences $\alpha_{0,k,p,i}$ and $\beta_{0,k,p,i}$ have the
same initial values,  they  coincide.
\end{proof}

\bibliographystyle{alpha}
\bibliography{./Biblio}

\begin{thebibliography}{BFMvH06}

\bibitem[AM22]{AmbMan22}
E.~Ambrosi and M.~Manzaroli.
\newblock Betti numbers of real semistable degenerations via real logarithmic
  geometry.
\newblock arXiv:2211.12134, 2022.

\bibitem[ARS21]{ArnRenSha18}
C.~Arnal, A.~Renaudineau, and K.~Shaw.
\newblock Lefschetz section theorems for tropical hypersurfaces.
\newblock {\em Ann. H. Lebesgue}, 4:1347--1387, 2021.

\bibitem[BB05]{BB-CombinatoricsCoxeterGroups}
A.~Björner and F.~Brenti.
\newblock {\em Combinatorics of Coxeter groups}, volume 231.
\newblock Springer Berlin Heidelberg, 2005.

\bibitem[BB06]{Br3}
B.~Bertrand and E.~Brugallé.
\newblock A {V}iro theorem without convexity hypothesis for trigonal curves.
\newblock {\em International Mathematics Research Notices}, 2006:Article ID
  87604, 33 pages, 2006.
\newblock doi:10.1155/IMRN/2006/87604.

\bibitem[BB07]{BerBih07}
B.~Bertrand and F.~Bihan.
\newblock Euler characteristic of real nondegenerate tropical complete
  intersections.
\newblock arXiv:0710.1222, 2007.

\bibitem[BB08]{Br4}
B.~Bertrand and E.~Brugall{\'e}.
\newblock A nonalgebraic patchwork.
\newblock {\em Math. Z.}, 259(3):481--486, 2008.

\bibitem[Ber10]{Ber2}
B.~Bertrand.
\newblock Euler characteristic of primitive {$T$}-hypersurfaces and maximal
  surfaces.
\newblock {\em J. Inst. Math. Jussieu}, 9(1):1--27, 2010.

\bibitem[BFMvH06]{BFMV06}
F.~Bihan, M.~Franz, C.~McCrory, and J.~van Hamel.
\newblock Is every toric variety an {M}-variety?
\newblock {\em Manuscripta Math.}, 120(2):217--232, 2006.

\bibitem[Bj{\"o}84]{Bjoe-PosetsRegularCw}
A.~Bj{\"o}rner.
\newblock Posets, regular {CW} complexes and bruhat order.
\newblock {\em European Journal of Combinatorics}, 5(1):7--16, 1984.

\bibitem[Bru22]{Bru22}
E.~Brugallé.
\newblock Euler characteristic and signature of real semi-stable degenerations.
\newblock {\em Journal of the Institute of Mathematics of Jussieu}, pages 1--8,
  2022.

\bibitem[BT82]{BotTu82}
R.~Bott and L.~W. Tu.
\newblock {\em Differential forms in algebraic topology}, volume~82 of {\em
  Grad. Texts Math.}
\newblock Springer, Cham, 1982.

\bibitem[Cha19]{Cha-ThinPosetsCw}
A.~Chandler.
\newblock Thin posets, {CW} posets, and categorification.
\newblock {\em ArXiv e-prints}, 2019.

\bibitem[Dan78]{Dan78}
V.~I. Danilov.
\newblock The geometry of toric varieties.
\newblock {\em Uspekhi Mat. Nauk}, 33(2(200)):85--134, 247, 1978.

\bibitem[DK86]{DanKho86}
V.~I. Danilov and A.~G. Khovanski\u{\i}.
\newblock Newton polyhedra and an algorithm for calculating {H}odge-{D}eligne
  numbers.
\newblock {\em Izv. Akad. Nauk SSSR Ser. Mat.}, 50(5):925--945, 1986.

\bibitem[DLRS10]{DeLoRamSan10}
J.~A. De~Loera, J.~Rambau, and F.~Santos.
\newblock {\em Triangulations}, volume~25 of {\em Algorithms and Computation in
  Mathematics}.
\newblock Springer-Verlag, Berlin, 2010.
\newblock Structures for algorithms and applications.

\bibitem[dLW98]{Loe}
J.~A. de~Loera and F.~J. Wickiln.
\newblock On the need of convexity in patchworking.
\newblock {\em Adv. in Appl. Math.}, 20:188--219, 1998.

\bibitem[DRHN19]{DiRHaaNil19}
S.~Di~Rocco, C.~Haase, and B.~Nill.
\newblock A note on discrete mixed volume and {H}odge-{D}eligne numbers.
\newblock {\em Adv. in Appl. Math.}, 104:1--13, 2019.

\bibitem[FMSS95]{FMSS-IntersectionTheorySpherical}
William {Fulton}, Robert {MacPherson}, F.~{Sottile}, and Bernd {Sturmfels}.
\newblock Intersection theory on spherical varieties.
\newblock {\em Journal of Algebraic Geometry}, 4(1):181--193, 1995.

\bibitem[FS97]{FulStu97}
W.~Fulton and B.~Sturmfels.
\newblock Intersection theory on toric varieties.
\newblock {\em Topology}, 36(2):335--353, 1997.

\bibitem[GKP94]{GKP94}
R.~L. Graham, D.~E. Knuth, and O.~Patashnik.
\newblock {\em Concrete mathematics}.
\newblock Addison-Wesley Publishing Company, Reading, MA, second edition, 1994.

\bibitem[GKZ94]{GKZ}
I.~M. Gelfand, M.~M. Kapranov, and A.~V. Zelevinsky.
\newblock {\em Discriminants, resultants, and multidimensional determinants}.
\newblock Mathematics: Theory \& Applications. Birkh\"auser Boston Inc.,
  Boston, MA, 1994.

\bibitem[Haa97]{Haa2}
B.~Haas.
\newblock Real algebraic curves and combinatorial constructions.
\newblock Thèse doctorale, Université de Strasbourg, 1997.

\bibitem[Hat02]{Hat02}
A.~Hatcher.
\newblock {\em Algebraic {T}opology}.
\newblock Cambridge University Press, Cambridge, 2002.

\bibitem[IKMZ19]{IKMZ19}
I.~Itenberg, L.~Katzarkov, G.~Mikhalkin, and I.~Zharkov.
\newblock Tropical homology.
\newblock {\em Math. Ann.}, 374(1-2):963--1006, 2019.

\bibitem[IS02]{IS}
I.~Itenberg and E.~Shustin.
\newblock Combinatorial patchworking of real pseudo-holomorphic curves.
\newblock {\em Turkish J. Math.}, 26(1):27--51, 2002.

\bibitem[IS03]{IS2}
I.~Itenberg and E.~Shustin.
\newblock Viro theorem and topology of real and complex combinatorial
  hypersurfaces.
\newblock {\em Israel J. Math.}, 133:189--238, 2003.

\bibitem[Ite95]{Ite95}
I.~Itenberg.
\newblock Counter-examples to {R}agsdale conjecture and {$T$}-curves.
\newblock In {\em Real algebraic geometry and topology ({E}ast {L}ansing, {MI},
  1993)}, volume 182 of {\em Contemp. Math.}, pages 55--72. Amer. Math. Soc.,
  Providence, RI, 1995.

\bibitem[Ite97]{Ite97}
I.~Itenberg.
\newblock Topology of real algebraic {$T$}-surfaces.
\newblock volume~10, pages 131--152. 1997.
\newblock Real algebraic and analytic geometry (Segovia, 1995).

\bibitem[IV96]{IV2}
I.~Itenberg and O.~Ya. Viro.
\newblock Patchworking algebraic curves disproves the {R}agsdale conjecture.
\newblock {\em Math. Intelligencer}, 18(4):19--28, 1996.

\bibitem[IV06]{IteVir06}
I.~Itenberg and O.~Viro.
\newblock Asymptotically maximal real algebraic hypersurfaces of projective
  space.
\newblock {\em Proceedings of {G}\"okova {G}eometry-{T}opology Conference},
  pages 91--105, 2006.

\bibitem[Jor98]{Jor-HomologyCohomologyToric}
A.~Jordan.
\newblock {\em Homology and cohomology of toric varieties}.
\newblock PhD thesis, Universität Konstanz, 1998.

\bibitem[JRS18]{JRS-Lefschetz11Theorem}
P.~Jell, J.~Rau, and K.~Shaw.
\newblock Lefschetz (1,1)-theorem in tropical geometry.
\newblock {\em Épijournal de Géométrie Algébrique}, 2(11), 2018.

\bibitem[OT92]{OT-ArrangementsHyperplanes}
P.~Orlik and H.~Terao.
\newblock {\em Arrangements of hyperplanes}, volume 300 of {\em Grundlehren der
  mathematischen Wissenschaften [Fundamental Principles of Mathematical
  Sciences]}.
\newblock Springer-Verlag, Berlin, 1992.

\bibitem[RRS]{RenRauSha22}
A.~Renaudineau, J.~Rau, and K.~Shaw.
\newblock Real phase structures and patchworking of tropical manifolds.
\newblock In preparation.

\bibitem[RRS22]{RenRauSha21}
A.~Renaudineau, J.~Rau, and K.~Shaw.
\newblock Real phase structures on matroid fans and matroid orientations.
\newblock {\em Journal of the London Mathematical Society}, n/a(n/a), 2022.

\bibitem[RS82]{RS-IntroductionPiecewiseLinear}
Colin~P. Rourke and Brian~J. Sanderson.
\newblock {\em Introduction to Piecewise-Linear Topology}, volume~69 of {\em
  Ergeb. Math. Grenzgeb.}
\newblock Springer-Verlag, Berlin, revised reprint of the 1972 original
  edition, 1982.

\bibitem[RS18]{RenSha18}
A.~Renaudineau and K.~Shaw.
\newblock Bounding the {B}etti numbers of real hypersurfaces near the tropical
  limit.
\newblock arXiv:1805.02030, 2018.

\bibitem[Stu94]{Stur94}
B.~Sturmfels.
\newblock Viro's theorem for complete intersections.
\newblock {\em Ann. Scuola Norm. Sup. Pisa Cl. Sci. (4)}, 21(3):377--386, 1994.

\bibitem[Vir84]{V1}
O.~Ya. Viro.
\newblock Gluing of plane real algebraic curves and constructions of curves of
  degrees {$6$} and {$7$}.
\newblock In {\em Topology (Leningrad, 1982)}, volume 1060 of {\em Lecture
  Notes in Math.}, pages 187--200. Springer, Berlin, 1984.

\bibitem[Vir06]{V2}
O.~Ya. Viro.
\newblock Patchworking real algebraic varieties.
\newblock \url{https://arxiv.org/abs/math/0611382}, 2006.

\bibitem[Voi02]{Voi02}
C~Voisin.
\newblock {\em Th\'{e}orie de {H}odge et g\'{e}om\'{e}trie alg\'{e}brique
  complexe}, volume~10 of {\em Cours Sp\'{e}cialis\'{e}s [Specialized
  Courses]}.
\newblock Soci\'{e}t\'{e} Math\'{e}matique de France, Paris, 2002.

\bibitem[Zha13]{Zha13}
I.~Zharkov.
\newblock The {Orlik}-{Solomon} algebra and the {Bergman} fan of a matroid.
\newblock {\em J. G{\"o}kova Geom. Topol. GGT}, 7:25--31, 2013.

\end{thebibliography}

\end{document}